\documentclass[12pt,a4paper,leqno,twoside]{article}
\usepackage{amsmath}
\usepackage{lipsum,appendix}
\usepackage{amsfonts}
\usepackage{amssymb}
\usepackage{graphicx}

\usepackage{hyperref}
\usepackage{color}
\usepackage{float}
\usepackage[tmargin=2.0cm, bmargin=2.0cm, lmargin=3.0cm, rmargin=3.0cm]{geometry}
\providecommand{\U}[1]{\protect\rule{.1in}{.1in}}
\newtheorem{theorem}{Theorem}[section]

\newtheorem{corollary}{Corollary}[section]

\newtheorem{definition}{Definition}[section]

\newtheorem{lemma}{Lemma}[section]

\newtheorem{proposition}{Proposition}[section]
\newtheorem{remark}{Remark}[section]

\newcommand{\halmos}{{\mbox{\, \vspace{3mm}}} \hfill
\mbox{$\Box$}}
\newenvironment{proof}[1][Proof]{\noindent\textbf{#1.} }{\halmos}
\numberwithin{equation}{section}

\newcommand{\ind}{1\hspace{-2.1mm}{1}}

\newcommand{\be}{\begin{equation}}
\newcommand{\ee}{\end{equation}}

\newcommand{\bq}{\begin{eqnarray}}
\newcommand{\eq}{\end{eqnarray}}

\newcommand{\WW}{\mathcal{W}}
\newcommand{\bo}{\boldsymbol}
\newcommand{\Lpq}{\bo{\Lambda}_{q}^{+}}
\newcommand{\Lmq}{\bo{\Lambda}_{q}^{-}}

\newcommand{\Lmgd}{\bo{\Lambda}_{\gamma_{0}+\delta}^{-}}
\newcommand{\Lpgd}{\bo{\Lambda}_{\gamma_{0}+\delta}^{+}}
\newcommand{\Lp}{\bo{\Lambda}^{+}}
\newcommand{\Lm}{\bo{\Lambda}^{-}}
\newcommand{\WWomega}{\WW^{(\omega)}}

\begin{document}
\author{Irmina Czarna$^a$, Adam Kaszubowski$^b$, Shu Li$^c$ and Zbigniew Palmowski$^a$\footnote{Corresponding author: \texttt{zbigniew.palmowski@pwr.edu.pl}}}
\title{Fluctuation identities for omega-killed Markov additive processes and dividend problem}
\date{{\small \today}}
\maketitle

\noindent \thanks{$^a$Faculty of Pure and Applied Mathematics, Hugo Steinhaus Centre, Wroc{\l}aw University of Science and Technology, Poland\\$^b$Mathematical Institute, University of Wroc{\l}aw, Poland
\\$^c$Department of Mathematics, University of Illinois at Urbana-Champaign, USA
}

\begin{abstract}
In this paper we solve the exit problems for an one-sided Markov additive process (MAP) which is exponentially killed with a bivariate
killing intensity $\omega(\cdot,\cdot)$ dependent on the present level of the process and the present state of the environment. Moreover, we analyze respective resolvents.
All identities are given in terms of new generalizations of classical scale matrices for the MAP.
We also remark on a number of applications of the obtained identities to (controlled) insurance risk processes.
In particular, we show that our results can be applied to the so-called Omega model,
where bankruptcy occurs at rate $\omega(\cdot,\cdot)$ when the surplus process becomes negative.
Finally, we consider the Markov modulated Brownian motion (MMBM) and present the results for the particular choice
of piecewise intensity function $\omega(\cdot,\cdot)$.

\vspace{3mm}


\end{abstract}
 \vspace{3mm}

\section{Introduction}
%The goal of this paper is to generalize the omega type two-sided exit problems into the MAP framework, i.e., from the one-dimensional result to the matrix form. This generalization will give us more flexibility in modeling the state-dependent features in risk modeling. Applications in the insurance field will be provided as well as some explicit examples.
In the fields of risk theory, financial mathematics, environmental problems, queueing and so forth, there are various applications of a Markov additive process (MAP) which
in continuous time is a natural generalization of a L\'evy process (see, e.g., \cite{AS03,LB12, DIKM10,DM09,IM10}).
%A Markov additive process (MAP) in continuous time is a natural generalization of a L\'evy process with various applications in risk theory, financial mathematics and queueing (see e.g.\cite{AS03}).
Furthermore, MAP can be seen as a L\'evy process in Markov environment, which provides
rich modeling possibilities. This paper solves exit problems for spectrally negative MAP which is exponentially killed with a bivariate
killing intensity $\omega(\cdot,\cdot)$ dependent on the present states of the process and the environment. Moreover, we analyze respective $\omega$-killed resolvents.
Recently, Li and Palmowski \cite{LP16} investigated $\omega$-killed exit identities and resolvents for a general (reflected)
spectrally negative L\'evy process. This paper generalized their results to Markov additive framework.

Before entering our discussion of this subject, we shall begin by defining the class of processes we intend to work with.
Let $(\Omega,\mathcal{F},\mathbf{F}, \mathbb{P})$ be a filtrated probability space, with filtration
$\mathbf{F} = \lbrace \mathcal{F}_t : t \geq 0 \rbrace$ which satisfies usual conditions.
Throughout this article, we will consider a bivariate process $(X,J)=\{(X_t,J_t)\}_{t\ge0}$ such that $X$ is a real-valued c\'adl\'ag $($right-continuous with left limits$)$
process and $J$ is a right-continuous jump process with a finite state space $E = \lbrace 1,2,...,N\rbrace$.
We say that $(X,J)$ is a MAP if, given $\lbrace J_t = i \rbrace$, the vector $(X_{t+s}-X_{t},J_{t+s})$
is independent of $\mathcal{F}_t$ and has the same law as $(X_s - X_0, J_s)$ given $\lbrace J_0 = i \rbrace$ for all $s,t \geq 0$ and $i \in E$.
Usually $X$ is called an additive component and $J$ is a background process representing the environment.
%It is common to say that $X$ is an additive component and $J$ is a background process representing the environment.
Moreover, we find the following representation of every MAP important. Straightforward conclusion from the definition gives that $J$ is a Markov chain.
Furthermore, one can observe that process $X$ evolves as some L\'evy process $X^{i}$ when $J$ is in state $i$.
In addition, when $J$ transits to state $j \neq i$, process $X$ jumps according to the distribution of the random variable $U_{ij}$, where $i,j \in E$. All above components are assumed to be independent.
The above structure explains why the another name for MAP is ``Markov-modulated L\'evy process".
Furthermore, let us remark that when $J$ lives on a single state, $X$ reduces to a L\'evy process.
Throughout this paper we assume that process $X$ has no positive jumps, thus $X^i$ is a spectrally negative L\'evy process and $U_{ij} \leq 0$ a.s.\ (for every $i,j \in E$).
We exclude the case when $X$ has monotone paths. We assume that $J$ is an irreducible Markov chain, with
$\bo{Q}$ being its transition probability matrix and $\bo{\pi}$ being its unique stationary vector.

One of the main contributions of this paper is the identification of the so-called $\omega$-scale matrices $\mathcal{W}^{(\omega)}$ and
$\mathcal{Z}^{(\omega)}$, which appear as the solutions of exit problems for MAPs with $\omega$-killing. Moreover, it is shown that these new generalizations of scale
matrices are solutions to some integral equations. In the case where the killing intensity $\omega$ is constant for every $i \in E$ and $x \in \mathbb{R}$, our results are consistent with the classical exit identities and resolvents obtained in \cite{IP12} and \cite{I14} in terms of the so-called scale matrices.

The paper is organized as follows. Section \ref{Prel} recalls some basic definitions and properties of MAPs and also introduces the
bivariate $\omega$-function and $\omega$-killed exit problems. In Section \ref{Main}, we formally define $\omega$-scale
matrices and present our main results. In Section \ref{Div}, we apply our results to find the value function for dividends paid until
ruin in the so-called Omega model. Section \ref{sec: examples} is dedicated to the analysis of some particular examples of $\omega$. Further numerical computations are provided. Finally, we postpone the existence of $\omega$-killed scale matrices as well as the proofs of the main results in Appendixes \ref{Lemma2.1} and \ref{Proof Main} for conciseness.%, we prove the existence of $\omega$-killed scale matrices and give the proofs of the main results.

\section{Preliminaries}\label{Prel}
In this section, we present some basic definitions and properties of MAPs. Let $\mathbf{F}(\alpha)$ be the matrix analogue of the Laplace exponent of the spectrally negative L\'evy process, namely
\[
\mathbb{E}\left( e^{\alpha X_t}, J_t=j|J_0=i\right)=\left(e^{\mathbf{F}(\alpha)t} \right)_{ij}, \quad \textrm{for $\alpha \geq 0$},
\]
which has an explicit representation
\[
\mathbf{F}(\alpha)=\text{diag}(\psi_1(\alpha),...,\psi_N(\alpha))+\mathbf{Q}\circ \mathbb{E}(e^{\alpha U_{ij}}).
\]
Recall that $\mathbf{Q}$ is the $N\times N$ transition rate matrix of $J$. Further we denote by $\psi_{i}$ the Laplace exponent of L\'evy process
$X^{i}$ when $J_t=i$ $($i.e., $\mathbb{E}( e^{\alpha X^i_t})=e^{\psi_{i}(\alpha)t}$$)$, and $\bo{A} \circ \bo{B}=(a_{ij}b_{ij})$ stands for entry-wise $($Hadamard$)$ matrix product.
Note that $\mathbf{F}(0)$ is the transition rate matrix of $J$, and hence our MAP is non-defective if and only if $\mathbf{F}(0) \vec{1}= \vec{{0}}$, where $\vec{{0}}$ and $\vec{1}$ denote the (column) vectors of $0$s and $1$s respectively (whereas the identity and the zero matrices are denoted by $\mathbf{I}$ and $\mathbf{0}$ respectively.)
Throughout this article, the law of $(X,J)$ such that $X_0 = x$ and $J_0 = i$ is denoted by $\mathbb{P}_{x,i}$ and its expectation by
$\mathbb{E}_{x,i}$. We will also use equivalently $\mathbb{E}_x[\cdot| J_0 = i]$ for $\mathbb{E}_{x,i}[\cdot]$ to emphasis the starting state. When $x =0$, we will write $\mathbb{P}(\cdot|J_0 = i)$ and $\mathbb{E}[\cdot|J_0  =i]$ respectively.
For a stopping time $\kappa$, the notation
$\mathbb{E}_x[\cdot,J_{\kappa}|J_0]$ is used to denote a $N \times N$ matrix whose $(i,j)$ entry equals to $\mathbb{E}_{x}[\cdot, J_{\kappa} = j|J_0=i]$.

In the study of exit problems of spectrally negative MAPs, the so-called scale matrices play an essential role, which can be defined analogously as the scale functions of spectrally negative L\'evy processes. First, let us define the first passage times:
\[
\tau_{x}^{+} = \inf \lbrace t > 0 : X_t \geq x \rbrace,
\qquad
\tau_{x}^{-} = \inf \lbrace t > 0 : X_t < x \rbrace.
\]
From Kyprianou and Palmowski \cite{KP08}, for $q \geq 0$, there exists a continuous, invertible matrix
function $\bo{W}^{(q)}:[0,\infty) \rightarrow \mathbb{R}^{N \times N}$ such that for all $0 \leq x \leq a$,
\begin{equation}\label{two-sided_W}
\mathbb{E}_x \Bigl[ e^{-q \tau_a^{+}}, \tau_a^{+} < \tau_0^{-}, J_{\tau_{a}^{+}}|J_0\Bigr] = \bo{W}^{(q)}(x)\bo{W}^{(q)}(a)^{-1}.
\end{equation}
Moreover, Ivanovs \cite{I11} and Ivanovs and Palmowski \cite{IP12} showed that $\bo{W}^{(q)}$ can be characterized by
\begin{equation}\label{matrix_W_def}
\widetilde{\mathbf{W}}^{(q)}(\alpha)=\left( \mathbf{F}(\alpha)-q\mathbf{I}\right)^{-1}, \quad \textrm{for large enough $\alpha$,}
\end{equation}
where $\widetilde{f}(\alpha)= \int_0^{\infty} e^{-\alpha x} f(x)  dx$ denotes the Laplace transform of the matrix function $f$.
Furthermore, the domain of $\bo{W}^{(q)}$ can extended to the negative half line by taking $\mathbf{W}^{(q)}(x) = \mathbf{0}$ for $x < 0$.
The basis of the above transform lies a probabilistic construction of the scale matrix $\mathbf{W}^{(q)}$ which
involves the first hitting time at level $x$ and can be written as
\begin{equation}
%\label{rep2_W}
\nonumber
\mathbf{W}^{(q)}(x)= e^{-\mathbf{\Lambda}^q x} \mathbf{L}^{q}(x),
\end{equation}
where $\mathbf{\Lambda}^q$ is the transition rate matrix of Markov chain $\{J_{\tau_x^+}\}_{x\geq 0}$, i.e.,
$\mathbb{P}(\tau_x^+<\textrm{e}_q, J_{\tau_x^+})=e^{\mathbf{\Lambda}^q x}$ with $\textrm{e}_q$ being an
independent exponential random variable of rate $q > 0$.
Moreover, $\mathbf{L}^q(x)$ is a matrix of expected occupation times at $0$ up to the first passage time over $x$.
In addition, the matrix $\mathbf{L}^q:=\mathbf{L}^q(\infty)$ is the expected occupation density at $0$ and it
is known that $\mathbf{L}^q$ has finite entries and is invertible unless the process is non-defective and $\bo{\pi}\mathbb{E}[X_1, J_1|J_0]\vec{1}=0$ (see \cite{IP12}).
Hence, we have
\begin{equation}\label{lim_W}
\lim_{x \to \infty}e^{\mathbf{\Lambda}^q x}\mathbf{W}^{(q)}(x)=
\lim_{x \to \infty}\mathbf{W}^{(q)}(x) e^{\mathbf{R}^q x}=\mathbf{L}^{q},
\end{equation}
where the matrix % $\mathbf{R}^q$ is given by the following relation:
$\mathbf{R}^q:=\left(\mathbf{L}^{q}\right)^{-1}\mathbf{\Lambda}^q \mathbf{L}^{q}.$
Moreover, it is easy to see that $\lim_{a \to \infty}\mathbf{W}^{(q)}(a)^{-1}=\mathbf{0}$,
since the Expectation (\ref{two-sided_W}) tends to $\mathbf{0}$ when $a \to \infty$, therefore, from the above argument,
\begin{equation}
%\label{lim_exp}
\nonumber
\lim_{x \to \infty}e^{\mathbf{\Lambda}^q x}= \lim_{x \to \infty}\mathbf{W}^{(q)}(x)^{-1}\mathbf{L}^{q}(x)=\mathbf{0}.
\end{equation}
The second scale matrix $\mathbf{Z}^{(q)}$ is then defined through the $\mathbf{W}^{(q)}$ matrix function:
\[
\mathbf{Z}^{(q)}(x)=\mathbf{I}-\int_0^{x}\mathbf{W}^{(q)}(y)dy \; (\mathbf{F}(0)-q\mathbf{I}).
\]
Note that $\mathbf{Z}^{(q)}(x)$ is continuous in $x$ with $\mathbf{Z}^{(q)}(0)=\mathbf{I}$.
Furthermore,
\begin{equation}
\nonumber
%\label{lim_Z}
\lim_{x \to \infty}e^{\mathbf{\Lambda}^q x}\mathbf{Z}^{(q)}(x)=\int_0^{\infty} e^{\mathbf{\Lambda}^q z} dz\mathbf{L}^{q} \; (q\mathbf{I}-\mathbf{F}(0)).
\end{equation}
\begin{remark}\rm
For the case without exponential killing ($q=0$), the upper subscript $q$ will be omitted in all above quantities, which write as $\mathbf{W}(x), \mathbf{Z}(x), \mathbf{L}(x), \mathbf{\Lambda}$, etc.
\end{remark}
For more details of the scale matrices, we refer the reader to \cite{I14, IP12}.
\begin{definition}
Let $\omega : E \times \mathbb{R} \rightarrow \mathbb{R}^{+}$ be a function defined as $\omega(i,x) = \omega_{i}(x)$,
where for a fixed $i \in E$, $\omega_i : \mathbb{R} \rightarrow \mathbb{R}^{+}$ is a bounded, nonnegative measurable function and its
value formulates the matrix $\bo{\omega}(x):=\text{diag}(\omega_1(x),...,\omega_N(x))$.
Let $\lambda>0$ be the upper bound of $|\omega_i(x)|$ on $[0,\infty)$ for all $i \in E$.
\end{definition}
%This functions will help us to generalize the killing of the process in the spirit of the paper of Li and Palmowski \cite{LP16},
%namely we investigate the following expectations.
Our main interest of this paper is deriving closed-form formulas for the occupation times (up to some exit times), weighted by the $\omega$ function defined above. More specifically, for $d\leq x\leq c$ and $1\le i,j \le N$, we are interested in the expectation matrices
whose $(i,j)$-th elements are, respectively,
\[
\mathbb{E}_{x} \left[e^{-\int_0^{\tau_c^+}\omega_{J_s}(X_s)ds}, \tau_c^+<\tau_d^-, J_{\tau_c^+}=j | J_0 = i \right] \text{ and } \mathbb{E}_{x} \left[e^{-\int_0^{\tau_d^-}\omega_{J_s}(X_s)ds},\tau_d^- < \tau_c^+, J_{\tau_d^-}=j | J_0 = i \right].
\]
Further discussions about applications with some particular $\omega$ will be presented in Section~\ref{sec: examples}.

\section{Main results}\label{Main}

\subsection{Omega scale matrices}

Before presenting our main results, we shall devote a little time to establishing some further notations.
Our main aim is to represent fluctuation identities for MAPs with $\omega$-killing in the terms of new $\omega$-scale matrices defined as the unique
solutions of the following equations:
\begin{align}
\mathcal{W}^{(\omega)}(x)&=\mathbf{W}(x)+ \mathbf{W}*\left(\boldsymbol{\omega}\mathcal{W}^{(\omega)}\right)(x), \label{W}\\
\mathcal{Z}^{(\omega)}(x)&=\mathbf{I}+ \mathbf{W}*\left(\boldsymbol{\omega}\mathcal{Z}^{(\omega)}\right)(x),\nonumber
\end{align}
where $f*g(x)= \int_0^{x} f(x-y) g(y) dy$ denotes the convolution of two matrix functions $f$ and~$g$.
The following lemma shows that the above $\omega$-scale matrices $\mathcal{W}^{(\omega)}$ and $\mathcal{Z}^{(\omega)}$ are well-defined and exist uniquely (see Appendix \ref{Lemma2.1} for the proof).
%Thus, we can present first representation of killed exit problems in the terms of $\WWomega$.
%For the proof we refer the reader to the Appendix \ref{Lemma2.1}.
%as well as prove of the existence and uniqueness of such functions.

\begin{lemma} \label{lem unique} For every $i,j \in E$, let us assume that ${h}_{ij}$ is a locally bounded function and $\omega_i$ is a bounded function on $\mathbb{R}$. There exists an unique solution to the following equation:
\begin{equation}
\mathbf{H}(x)=\mathbf{h}(x)+\mathbf{W}*\left( \boldsymbol{\omega}\mathbf{H}\right)(x),\label{Hh}
\end{equation}
where $\mathbf{H}(x)=\mathbf{h}(x)$ for $x<0$. Furthermore, for any fixed $\delta>0$, $\mathbf{H}$ satisfies (\ref{Hh}) if and only if $\mathbf{H}$ satisfies:
\begin{equation}
\mathbf{H}(x)=\mathbf{h}_{\delta}(x)+ \mathbf{W}^{(\delta)}*\left((\boldsymbol{\omega}-\delta\mathbf{I})\mathbf{H}\right)(x), \label{Hhdelta}
\end{equation}
where $\mathbf{h}_{\delta}(x)=\mathbf{h}(x)+\delta \mathbf{W}^{(\delta)}*\mathbf{h}(x)$.
 \end{lemma}

We further introduce more general scale matrices $\mathcal{W}^{(\omega)}(x,y)$ and $\mathcal{Z}^{(\omega)}(x,y)$ to allow shifting:
\begin{align}
\mathcal{W}^{(\omega)}(x,y)&=\mathbf{W}(x-y)+\int_{y}^{x} \mathbf{W}(x-z) \boldsymbol{\omega}(z)\mathcal{W}^{(\omega)}(z,y)dz, \label{Woxy}\\
\mathcal{Z}^{(\omega)}(x,y)&=\mathbf{I}+\int_{y}^{x} \mathbf{W}(x-z)\boldsymbol{\omega}(z)\mathcal{Z}^{(\omega)}(z,y)dz. \label{Zoxy}
\end{align}

Also note that $\mathcal{W}^{(\omega)}(x,0)=\mathcal{W}^{(\omega)}(x)$, $\mathcal{Z}^{(\omega)}(x,0)=\mathcal{Z}^{(\omega)}(x)$, as well as
\begin{equation}
\mathcal{W}^{(\omega^*)}(x-y)=\mathcal{W}^{(\omega)}(x,y),\quad \mbox{and}\quad \mathcal{Z}^{(\omega^*)}(x-y)=\mathcal{Z}^{(\omega)}(x,y), \label{shiftw}
\end{equation}
with $\omega^*(\cdot,z)=\omega(\cdot,z+y)$, and  %Therefore, using the shifting arguments, we are able to generalize the exit identities in Theorems \ref{thm1} and \ref{thm2} to the following corollary.

Since $ \mathbf{W}^{(\delta)}- \mathbf{W}= \delta \mathbf{W}^{(\delta)}* \mathbf{W}$ and $ \mathbf{Z}^{(\delta)}- \mathbf{Z}=
\delta \mathbf{W}^{(\delta)}* \mathbf{Z}$, it is easy to check that
\begin{align}
\mathcal{W}^{(\omega)}(x,y)&=\mathbf{W}^{(\delta)}(x-y)+\int_{y}^{x} \mathbf{W}^{(\delta)}(x-z) (\boldsymbol{\omega}(z)-\delta\mathbf{I})\mathcal{W}^{(\omega)}(z,y)dz, \label{Wxyd}\\
\mathcal{Z}^{(\omega)}(x,y)&=\mathbf{Z}^{(\delta)}(x-y)+\int_{y}^{x} \mathbf{W}^{(\delta)}(x-z)(\boldsymbol{\omega}(z)-\delta\mathbf{I})\mathcal{Z}^{(\omega)}(z,y)dz.\label{Wxyd2}
\end{align}

To solve the one-sided upward problem (i.e., to get Corollary \ref{One-sided} (i)) we have to assume additionally that
\begin{equation}\label{omega_for_corollary}
\omega_i(x) \equiv \beta \geq 0, \quad \textrm{for all } x\leq 0 \textrm{ and } i \in E.
\end{equation}
Hence we define a matrix function $\mathcal{H}^{(\omega)}$ which satisfies the following integral equation
\begin{equation}\label{matrixH}
\mathcal{H}^{(\omega)}(x)= e^{-\mathbf{R}^{\beta}x}+\int_{0}^{x} \mathbf{W}^{(\beta)}(x-z) (\boldsymbol{\omega}(z)-\beta\mathbf{I})\mathcal{H}^{(\omega)}(z)dz.
\end{equation}

%At the beginning we will prove the lemma which will provide existence and uniqueness of the matrix $\omega$-scale function.

%We first introduce the omega function $\omega_{J_t}(X_t)$, which is a function of the current surplus level $X_t$ given the current state $J_t$. It could be interpreted as the state-dependent (regime-switching) interest (or discount) rates or ruin rates. For example,
%\begin{itemize}
%
%\item $\omega_i(x)=q_i+p_i1_{\{x\in (a,b)\}}$. It allows to calculate the occupation time (in different states or some interval) before exiting.
%%: consider the $q$ is the percentage of investment in the bond?
%
%Especially, if $p_i=0$, $q_i$ can represent the state-dependent interest rate.
%
%\item $\omega_i(x)>0$ for $x<0$ could represent the omega ruin rate.
%
%
%\end{itemize}
%
% Denote $\boldsymbol{\omega}(x)=diag\{\omega_i(x)\}_{i=1}^{n}$.
%
%Define the first passage times
%\[
%\tau_x^{+(-)}=\inf\{t\ge0: X_t \ge(\le)x\}.
%\]
%
%There are two main quantities of interest.
%In the next lemma, we will show the existence and uniqueness of the (matrix) $\omega$-scale function.

\subsection{Exit problems and resolvents}
In this section, we establish our main results of fluctuation identities and resolvents for spectrally negative $\omega$-killed MAPs.
The proofs of the below theorems and corollary are postponed to Appendix \ref{Proof Main}, since the arguments tend to be technical,
and the results intuitively hold in a similar manner as presented in \cite{LP16}.

\begin{theorem} \textbf{(Two-sided exit problem)}\label{Two-sided}\\
For invertible matrix function $\mathcal{W}^{(\omega)}$ and for
$\mathcal{Z}^{(\omega)}$ given in (\ref{Woxy}) and (\ref{Zoxy}) respectively, the following hold:
\begin{itemize}
\item[(i)] For $d \leq x \leq c$,
\begin{equation*}
\mathbf{A}^{(\omega)}_d(x,c):=\mathbb{E}_{x} \left[e^{-\int_0^{\tau_c^+}\omega_{J_s}(X_s)ds}, \tau_c^+<\tau_d^-, J_{\tau_c^+}|J_0 \right]
=\mathcal{W}^{(\omega)}(x,d) \mathcal{W}^{(\omega)}(c,d)^{-1}.
\end{equation*}
\item[(ii)] For $d \leq x \leq c$, \begin{align*}
\mathbf{B}^{(\omega)}_d(x,c)&:=\mathbb{E}_{x} \left[e^{-\int_0^{\tau_d^-}\omega_{J_s}(X_s)ds},\tau_d^- < \tau_c^+, J_{\tau_d^-}|J_0 \right]\\
&\ =\mathcal{Z}^{(\omega)}(x,d)-\mathcal{W}^{(\omega)}(x,d) \mathcal{W}^{(\omega)}(c,d)^{-1}\mathcal{Z}^{(\omega)}(c,d).
\end{align*}
\end{itemize}
\end{theorem}

\begin{remark}\rm
When $d=0$, we use simplified notations:
$\mathbf{A}^{(\omega)}(x,c):=\mathbf{A}^{(\omega)}_0(x,c)$ and $\mathbf{B}^{(\omega)}(x,c):=\mathbf{B}^{(\omega)}_0(x,c)$.
\end{remark}

Now, taking the limits $d \to -\infty$ and $c \to \infty$ (as well as $d=$0) in Theorem \ref{Two-sided} (i) and (ii) respectively, we obtain the following corollary.

\begin{corollary} \textbf{(One-sided exit problem)}\label{One-sided}
\begin{itemize}
\item[(i)] Under the assumption (\ref{omega_for_corollary}), for $ x \leq c$, \begin{equation*}
%\mathbf{A}^{(\omega)}_{\infty}(x,c):=
\mathbb{E}_{x} \left[e^{-\int_0^{\tau_c^+}\omega_{J_s}(X_s)ds}, \tau_c^+<\infty, J_{\tau_c^+}|J_0 \right]
=\mathcal{H}^{(\omega)}(x) \mathcal{H}^{(\omega)}(c)^{-1},
\end{equation*}
for invertible matrix function $\mathcal{H}^{(\omega)}$ given in (\ref{matrixH}).
\item[(ii)] For $x \geq 0 $ and $\lambda>0$,
\begin{equation*}
%\mathbf{B}^{(\omega)}(x):=
\mathbb{E}_{x} \left[e^{-\int_0^{\tau_0^-}\omega_{J_s}(X_s)ds},\tau_0^- < \infty, J_{\tau_0^-}|J_0 \right]
=\mathcal{Z}^{(\omega)}(x)- \mathcal{W}^{(\omega)}(x)\mathbf{C}_{\mathcal{W}(\infty)^{-1}\mathcal{Z}(\infty)},
\end{equation*}
where matrix
\begin{equation*}
 \mathbf{C}_{\mathcal{W}(\infty)^{-1}\mathcal{Z}(\infty)}:=
 %\\&
 \lim_{c \to \infty} \mathcal{W}^{(\omega)}(c)^{-1}\mathcal{Z}^{(\omega)}(c)
 %\left(\mathbf{L}^{\lambda}+\int_{0}^{\infty} e^{\mathbf{\Lambda}^{\lambda} z}
%\mathbf{L}^{\lambda}(\boldsymbol{\omega}(z)-\lambda\mathbf{I})\mathcal{W}^{(\omega)}(z) dz\right)^{-1}
%\int_{0}^{\infty} e^{\mathbf{\Lambda}^{\lambda} z} \mathbf{L}^{\lambda}(\boldsymbol{\omega}(z)-\lambda\mathbf{I})\mathcal{Z}^{(\omega)}(z) dz
\end{equation*}
%where matrix $\mathbf{C}_{\mathcal{W}(\infty)^{-1}\mathcal{Z}(\infty)}:= \lim_{c \to \infty} \mathcal{W}^{(\omega)}(c)^{-1}\mathcal{Z}^{(\omega)}(c)$
exists and has finite entries.
\end{itemize}
\end{corollary}

Next, we present the representation of $\omega$-type resolvents.

\begin{theorem} \textbf{(Resolvents)}\label{Res}
\begin{itemize}
\item[(i)] For $d\le x \leq c$,
 \begin{align*}
\boldsymbol{U}^{(\omega)}_{(d,c)}(x,dy)&:=\int_0^{\infty} \mathbb{E}_{x}\left[\exp\left( -\int_0^{t} \omega_{J_s}(X_s)ds \right),  X_t\in dy,  t<\tau_d^- \wedge
\tau_c^+, J_t|J_0 \right] dt\\
&\ = \left ( \mathcal{W}^{(\omega)}(x,d) \mathcal{W}^{(\omega)}(c,d)^{-1} \mathcal{W}^{(\omega)}(c,y)-\mathcal{W}^{(\omega)}(x,y)\right) dy.
\end{align*}

\item[(ii)] For $ x \geq 0$ and $\lambda>0$ , \begin{align*}
\boldsymbol{U}^{(\omega)}_{(0,\infty)}(x,dy)&:=\int_0^{\infty} \mathbb{E}_{x}\left[\exp\left( -\int_0^{t} \omega_{J_s}(X_s)ds \right),  X_t\in dy,  t<\tau_0^-, J_t|J_0 \right] dt\\
&\ = \left ( \mathcal{W}^{(\omega)}(x) \mathbf{C}_{\mathcal{W}(\infty)^{-1}\mathcal{W}(\infty)}(y) -\mathcal{W}^{(\omega)}(x,y)\right) dy,
\end{align*}
where
\begin{equation*}
\mathbf{C}_{\mathcal{W}(\infty)^{-1}\mathcal{W}(\infty)}(y):=\lim_{c \to \infty} \mathcal{W}^{(\omega)}(c)^{-1}\mathcal{W}^{(\omega)}(c,y)
%&\left(\mathbf{L}^{\lambda}+\int_{0}^{\infty} e^{\mathbf{\Lambda}^{\lambda} z}
%\mathbf{L}^{\lambda}(\boldsymbol{\omega}(z)-\lambda\mathbf{I})\mathcal{W}^{(\omega)}(z) dz\right)^{-1} \\&
%\times \left(e^{\mathbf{\Lambda}^{\lambda} y}\mathbf{L}^{\lambda}+\int_{y}^{\infty} e^{\mathbf{\Lambda}^{\lambda} z} %\mathbf{L}^{\lambda}(\boldsymbol{\omega}(z)-\lambda\mathbf{I})\mathcal{W}^{(\omega)}(z,y) dz\right)
\end{equation*}
is well defined and finite matrix.
%$\mathbf{C}_{\mathcal{W}(\infty)^{-1}\mathcal{W}(\infty)}(y)= \lim_{c \to \infty} \mathcal{W}^{(\omega)}(c)^{-1}\mathcal{W}^{(\omega)}(c,y)$.

\item[(iii)] For $ x,y\le c$, \begin{align*}
\boldsymbol{U}^{(\omega)}_{(-\infty,c)}(x,dy):&=\int_0^{\infty} \mathbb{E}_{x}\left[\exp\left( -\int_0^{t} \omega_{J_s}(X_s)ds \right),  X_t\in dy,  t< \tau_c^+, J_t|J_0 \right] dt\\
&= \left ( \mathcal{H}^{(\omega)}(x) \mathcal{H}^{(\omega)}(c)^{-1} \mathcal{W}^{(\omega)}(c,y)-\mathcal{W}^{(\omega)}(x,y)\right) dy.
\end{align*}

\item[(iv)] For $x \in \mathbb{R}$, %and $\lambda>0$
%For $\omega_i(x)\equiv q$ for all $i \in E$ and for  $x\geq M$ for some $M,q\geq 0$,
\begin{align*}
\boldsymbol{U}^{(\omega)}_{(-\infty,\infty)}(x,dy):&=\int_0^{\infty} \mathbb{E}_{x}\left[\exp\left( -\int_0^{t} \omega_{J_s}(X_s)ds \right),  X_t\in dy, J_t|J_0 \right] dt\\
&= \left ( \mathcal{H}^{(\omega)}(x) \mathbf{C}_{\mathcal{H}(\infty)^{-1}\mathcal{W}(\infty)}(y)-\mathcal{W}^{(\omega)}(x,y)\right) dy,
\end{align*}
where  matrix $\mathbf{C}_{\mathcal{H}(\infty)^{-1}\mathcal{W}(\infty)}= \lim_{c \to \infty} \mathcal{H}^{(\omega)}(c)^{-1}\mathcal{W}^{(\omega)}(c,y)$
exists and has finite entries.
\end{itemize}
\end{theorem}

\section{Dividends in the Omega ruin model}\label{Div}
In this section, we present one application of the previously obtained results on dividend problem.
The optimal dividend problem is very popular in the field of applied mathematics. De Finetti \cite{DF57} was the first who introduced the
dividend model in risk theory. He proposed the model in which company's surplus is described by random walk with increments $\pm 1$.
In his work, it was proved that, under the rule of maximization of expected discounted dividends before the classical ruin occurs (the surplus reaches below level $0$), the optimal strategy is the so-called barrier strategy %with assumption that dividends are paying before process hits level $0$, so called classical ruin time.
which is described as follows. For a fixed level $c > 0$, whenever the surplus process reaches this level, one reflects the process
and pays all funds above $c$ as dividends. In the literature, there is a rich set of articles in which this problem was studied in the continuous time;
see, e.g., Loeffen \cite{RL08}, Loeffen and  Renaud \cite{LRJFR10} and Avram et al. \cite{APP07} where the value function of
the barrier strategy and the optimal barrier level was described in the terms of the scale functions.
%\footnote{\textcolor{red}{Add much more references here.}}

In this paper, we
assume that the company's reserve process is governed by a Markov additive process $(X,J)$.
Moreover, we assume that this company pays dividends according to the barrier strategy until omega ruin time defined in the following way.
Fix an exponential random variable $\textrm{e}_1$ (with mean $1$) and level $-d \leq 0$, and then omega ruin time is defined as
\[
\tau_{\omega}^d=\inf\{t\ge 0:  \int_0^{t}\omega_{J_s}(X_s) dx>\textrm{e}_1 \text{ or } X_t <-d\},
\]
where, for all $i \in E$, $\omega_i(x) \geq 0$ for $x \geq -d$ and $\omega_i(x) = 0$ when $x < -d$ . Thus ruin can occur in two situations. The first is the situation in which the process crosses a fixed level $-d
\leq 0$ $($for $d =0$ we have a case of classical ruin time$)$. The second possibility is when bankruptcy happens in the so-called red zone and the intensity of this bankruptcy
is a function of current level of the additive component $X$ and the Markov chain $J$. For more details related to this omega ruin time, we refer to \cite{GSY12} and \cite{LP16}.

Immediately from the definition of $\tau_{\omega}^{d}$, one can conclude that
\[
\mathbb{P}_x(\tau_{\omega}^d>t)=\mathbb{E}_x\left [ e^{-\int_0^{t} \omega_{J_s}(X_s) ds}, \tau_{-d}^- >t\right].
\]
We denote dividend barrier strategy $($at $c)$ $\pi^{c}$ as follows
\begin{equation}
\nonumber
\pi^c = \lbrace L_s^c : t \geq 0 \rbrace,
\end{equation}
which is a non-decreasing, left-continuous $\bo{F}$-adapted process starting at zero. Random variable $L_t^c$ can be interpreted as the cumulative dividends paid up to time $t$. In the case of the barrier strategy, we have
\[
L_t^c = \sup_{s \leq t}[X_s - c]\vee 0.
\]
In the following theorem, we set $d=0$ and then consider the general $d$ in the corollary.

\begin{theorem} \label{thm41}
Assume that dividends are discounted at a constant force of interest $\delta>0$ and $d=0$.
The expected discounted present value of the dividends paid before omega ruin ($\tau_{\omega}:=\tau_{\omega}^0$)
under a constant dividend barrier $c$ is given by
\begin{equation*}%\label{vc}
\mathbf{v}_c(x):=\mathbb{E}_x\left[ \int_0^{\tau_{\omega}} e^{-\delta t}  dL_t, J_{\tau_{\omega}}| J_0\right]=
\begin{cases} \mathcal{W}^{(\delta+\omega)}(x)\mathcal{W}^{(\delta+\omega)\prime}(c)^{-1}, & \textrm{for} \quad  0 < x \leq c, \\ (x-c) + \mathcal{W}^{(\delta+\omega)}(c)\mathcal{W}^{(\delta+\omega)\prime}(c)^{-1}, & \textrm{for} \quad x > c, \end{cases}
\end{equation*}
for the invertible matrix function:
$$\mathcal{W}^{(\delta+\omega)\prime}(c)=\mathbf{W}^{\prime}(c)+\int_0^c \mathbf{W}^{\prime}(c-y)(\mathbf{\omega}(y)+\delta\mathbf{I})\mathcal{W}^{(\delta+\omega)}(y) dy
+\mathbf{W}(0)(\mathbf{\omega}(c)+\delta\mathbf{I})\mathcal{W}^{(\delta+\omega)}(c). $$
\end{theorem}
\begin{proof}
At the beginning we will treat the case of  $0 < x \leq c$.
Conditioning on reaching the level $c$ first, we have
\[
\mathbf{v}_c(x)=\mathbf{A}^{(\omega)}(x,c)\mathbf{v}_c(c)=\mathcal{W}^{(\delta+\omega)}(x)\mathcal{W}^{(\delta+\omega)}(c)^{-1}\mathbf{v}_c(c).
\]
As a first step we will find a lower bound for $\mathbf{v}_c(c)$. For $m\in \mathbb{N}$, consider that the dividend is not paid until reaching the level $c+\frac{1}{m}$:
\begin{align*}
\mathbf{v}_c (c) \ge & \ \mathbb{E}_c\left [e^{-\int_0^{\tau_{c+1/m}^{+}}(\delta+\omega_{J_s}(X_s))ds} , \tau_{c+\frac{1}{m}}^+<\tau_0^{-}, J_{\tau_{c+\frac{1}{m}}^{+}}|J_{0} \right]
 \mathbf{v}_c\left(c+\frac{1}{m}\right)\\
=& \mathbb{E}_c\left [e^{-\int_0^{\tau_{c+1/m}^{+}}(\delta+\omega_{J_s}(X_s))ds} , \tau_{c+\frac{1}{m}}^+<\tau_0^{-}, J_{\tau_{c+\frac{1}{m}}^{+}}|J_{0}\right]
\left ( \mathbf{v}_c(c)+\frac{1}{m}\mathbf{I}\right),
\end{align*}
where the last equality is due to the dividend of $\frac{1}{m}$ paid immediately and the fact that the drop in surplus will not cause the state transition.

On the other hand, an upper bound can be found as
\begin{align*}
\mathbf{v}_c (c) \leq & \ \mathbb{E}_c\left [e^{-\int_0^{\tau_{c+1/m}^{+}}(\delta+\omega_{J_s}(X_s))ds} , \tau_{c+\frac{1}{m}}^+<\tau_0^{-}, J_{\tau_{c+\frac{1}{m}}^{+}}|J_0 \right] \left( \mathbf{v}_c(c)+\frac{1}{m}\mathbf{I}\right)\\
& +\frac{1}{m}\mathbb{E}_{c}\left[ \int_0^{\tau_{c+1/m}^{+}} e^{-\delta t} dt  \ e^{-\int_0^{\tau_{c+1/m}^{+}}\omega_{J_s}(X_s)ds} , \tau_{c+\frac{1}{m}}^{+}<\tau_0^- , J_{\tau_0^-} | J_0\right]\\
&+\mathbb{E}_{c}\left[ \int_0^{\tau_{\omega}} e^{-\delta t} dL_{t}^c, \tau_{\omega} <\tau_{c+\frac{1}{m}}^{+} , J_{\tau_{\omega}} | J_0\right],
\end{align*}
where $L_t^{c}$ will be bounded by $\frac{1}{m}$ for the process starting from level $c$ to level $c+\frac{1}{m}$, i.e.,
\[
\mathbb{E}_{c}\left[ \int_0^{\tau_{\omega}} e^{-\delta t} dL_{t}^c, \tau_{\omega} <\tau_{c+\frac{1}{m}}^{+} , J_{\tau_{\omega}} | J_0\right]
 \le \frac{1}{m} \mathbb{P}_c\left (\tau_{\omega} <\tau_{c+\frac{1}{m}}^{+} , J_{\tau_{\omega}}| J_0\right).
\]

Note that as $m\rightarrow \infty$, the following two limits approach to $ \mathbf{0}$:
\[
\lim_{m\rightarrow \infty} \mathbb{E}_{c}\left[ \int_0^{\tau_{c+1/m}^{+}} e^{-\delta t} dt
\ e^{-\int_0^{\tau_{c+1/m}^{+}}\omega_{J_s}(X_s)ds} , \tau_{c+\frac{1}{m}}^{+}<\tau_0^- , J_{\tau_0^-} | J_0\right]= \mathbf{0},
\]
and
\[
\lim_{m\rightarrow \infty} \mathbb{P}_c\left (\tau_{\omega} <\tau_{c+\frac{1}{m}}^{+} , J_{\tau_{\omega}} | J_0\right) =\mathbf{0}.
\]
See Renaud and Zhou \cite{RZ07} and Czarna et al. \cite{CLPZ16} for more details.

Therefore, by the upper and lower bounds,
\begin{align*}
\mathbf{v}_c (c)& = \mathbb{E}_c\left [e^{-\int_0^{\tau_{c+1/m}^{+}}(\delta+\omega_{J_s}(X_s))ds} , \tau_{c+\frac{1}{m}}^+<\tau_0^{-}, J_{\tau_{c+\frac{1}{m}}^{+}}|J_0 \right]
\left( \mathbf{v}_c(c)+\frac{1}{m}\mathbf{I}\right)+o\left(\frac{1}{m}\right)\\
&= \mathcal{W}^{(\delta+\omega)}(c)\mathcal{W}^{(\delta+\omega)}(c+\frac{1}{m})^{-1} \left( \mathbf{v}_c(c)+\frac{1}{m}\mathbf{I}\right)+o\left(\frac{1}{m}\right),
\end{align*}
and hence
$$\left( \mathbf{I} -  \mathcal{W}^{(\delta+\omega)}(c)\mathcal{W}^{(\delta+\omega)}(c+\frac{1}{m})^{-1}\right)\mathbf{v}_c (c)=  \frac{1}{m}
\mathcal{W}^{(\delta+\omega)}(c) \mathcal{W}^{(\delta+\omega)}(c+\frac{1}{m})^{-1}+o\left(\frac{1}{m}\right),$$
 $$\frac{1}{m}\left(   \mathcal{W}^{(\delta+\omega)}(c+\frac{1}{m})\mathcal{W}^{(\delta+\omega)}(c)^{-1} - \mathbf{I}  \right) \mathbf{v}_c (c)=
 \mathbf{I}+ o\left(\frac{1}{m}\right),$$
%=\left(  \mathcal{W}^{(\delta+\omega)}(c+\frac{1}{m})\mathcal{W}^{(\delta+\omega)}(c)^{-1} - \mathcal{W}^{(\delta+\omega)}(c) \mathcal{W}^{(\delta+\omega)}(c)^{-1} \right) ^{-1}
%\frac{1}{m}  +o\left(\frac{1}{m}\right)\\
$$\left( \frac{ \mathcal{W}^{(\delta+\omega)}(c+\frac{1}{m}) - \mathcal{W}^{(\delta+\omega)}(c) }{1/m} \right)\mathcal{W}^{(\delta+\omega)}(c)^{-1} \mathbf{v}_c (c)=\mathbf{I}+ o\left(\frac{1}{m}\right).
$$
Letting $m\rightarrow\infty$, it turns out
\[  \mathcal{W}^{(\delta+\omega)\prime}(c) \mathcal{W}^{(\delta+\omega)}(c)^{-1} \mathbf{v}_c (c)=\mathbf{I},
\]
where matrix $$\mathcal{W}^{(\delta+\omega)\prime}(c)=\mathbf{W}^{\prime}(c)+\int_0^c \mathbf{W}^{\prime}(c-y)(\mathbf{\omega}(y)+\delta\mathbf{I})\mathcal{W}^{(\delta+\omega)}(y) dy
+\mathbf{W}(0)(\mathbf{\omega}(c)+\delta\mathbf{I})\mathcal{W}^{(\delta+\omega)}(c) $$
is well-defined since the scale matrix $\mathbf{W}$ is almost everywhere differentiable, see \cite{KP08}.
Furthermore, one can observe that, from representation (\ref{Sub-stoch}), the above matrix is invertible for any $c > 0$ and
then $  \mathbf{v}_c (c)= \WW^{(\delta+\omega)}(c)\WW^{(\delta+\omega)\prime}(c)^{-1} .$

To end this proof, note that for $x > c$, one is immediately paying dividend of size $x-c$ (and this will not cause the state transition), therefore
$$
\mathbf{v}_c(x) = (x-c) + \mathbf{v}_c(c) = (x-c)+ \WW^{(\delta+\omega)}(c)\WW^{(\delta+\omega)\prime}(c)^{-1}.
$$
\end{proof}

Making use of Theorem \ref{thm41} and the shifting argument, we can state the representation for value function for a general $d \ge 0$.
\begin{corollary}
For $\delta>0$, the expected present value of the dividend paid before omega ruin ($\tau_{\omega}^d$) under a constant dividend barrier $c$ is
\begin{equation*}\label{vcd}
\mathbf{v}^d_c(x):=\mathbb{E}_x\left[ \int_0^{\tau_{\omega}^d} e^{-\delta t}  dL_t, J_{\tau_{\omega}^d}| J_0\right]=
\begin{cases} \mathcal{W}^{(\delta+\omega)}(x,-d)\mathcal{W}^{(\delta+\omega)\prime}(c,-d)^{-1} \qquad \textrm{for} \quad -d < x \leq c,
 \\ (x-c) + \mathcal{W}^{(\delta+\omega)}(c,-d)\mathcal{W}^{(\delta+\omega)\prime}(c,-d)^{-1} \quad \textrm{for} \quad x > c. \end{cases}
\end{equation*}
for invertible matrix:
\begin{align*}
\mathcal{W}^{(\delta+\omega)\prime}(c,-d)=&\mathbf{W}^{\prime}(c+d)+\int_{-d}^c \mathbf{W}^{\prime}(c-y)(\mathbf{\omega}(y)+\delta\mathbf{I})
\mathcal{W}^{(\delta+\omega)}(y,-d) dy \\
&+\mathbf{W}(0)(\mathbf{\omega}(c)+\delta\mathbf{I})\mathcal{W}^{(\delta+\omega)}(c,-d).
\end{align*}
\end{corollary}

\section{Examples} \label{sec: examples}
The aim of this section is to give examples of $\omega$-scale matrices when the $\omega$ function is specified.
 We would like to present relations between $\WWomega$ and $\bo{W}^{(q)}$, for some $q \geq 0$, as well as numerical examples which
 help to understand better the nature of explored matrix-valued functions. We will start with short analyse of $\bo{W}^{(q)}$
for Markov modulated Brownian motion, since this model will be a base for more complicated scale matrices.
\subsection{Markov modulated Brownian motion}\label{MMBM_sec}
In this part, we will consider a special case when $(X,J)$ is a Markov modulated Brownian motion.
Our aim is to derive some relations which will be useful in the subsequent examples. Let $X_i$ be a Brownian motion with variance $\sigma_i^2>0$ and
 drift $\mu_i$ for all $i \in E$. Further denote $\boldsymbol{\sigma}$ and $\boldsymbol{\mu}$ as the (column) vectors of $\sigma_i$ and $\mu_i$, and $\Delta_{\boldsymbol{v}}$ as the
diagonal matrix with $\boldsymbol{v}$ on the diagonal. Therefore, the matrix Laplace exponent $\bo{F}(s)$ is given by
\[\bo{F}(s)=\frac{1}{2}\Delta_{\boldsymbol{\sigma}}^2 s^2 +\Delta_{\boldsymbol{\mu}}s +\boldsymbol{Q}.
\]

Despite the case when ${\kappa} := \bo{\pi} \bo{\mu} = 0$ and $q =0$, Ivanovs \cite{I10} gives the representation of the $q$-scale matrix
\begin{equation}\label{MMBM Scale matrix formula}
\mathbf{W}^{(q)}(x)=\left( e^{-\mathbf{\Lambda}^{+}_q x}-e^{\mathbf{\Lambda}^{-}_q x}\right)\bo{\Xi}_q,
\end{equation}
where $\bo{\Xi}_q^{-1}=-\frac{1}{2} \Delta_{\boldsymbol{\sigma}}^2(\Lpq +\Lmq)$ and $\bo{\Lambda}^{\pm}_q$ are the (unique)
right solutions to the matrix integral equation $\bo{F}(\mp\mathbf{\Lambda}^{\pm}_q)=\mathbf{0}$, that is,
\begin{equation}\label{Characteristic equation}
\Delta_{\frac{\bo{\sigma}^2}{2}}(\bo{\Lambda}^{\pm}_q)^{2} \mp \Delta_{\bo{\mu}}\bo{\Lambda}^{\pm}_q + \Bigl(\bo{Q} - q \textbf{I}\Bigr) = \bo{0}.
\end{equation}
%\begin{equation}\label{Characteristic equation}
%\begin{split}
%&\Delta_{\frac{\bo{\sigma}^2}{2}}(\Lpq)^{2} - \Delta_{\bo{\mu}}\Lpq + \Bigl(\bo{Q} - q \textbf{I}\Bigr) = \bo{0}, \\
%& \Delta_{\frac{\bo{\sigma}^2}{2}}(\Lmq)^{2} + \Delta_{\bo{\mu}}\Lmq + \Bigl(\bo{Q} - q \textbf{I}\Bigr) = \bo{0}.
%\end{split}
%\end{equation}
In the next lemma, we present relations between $\Lpq$ and $\Lmq$.
\begin{lemma}
For $q \geq 0$, we have
\begin{equation}
%\label{LambdaPlusGamma minus Cgamma}
\label{Relation C Lp}
\Delta_{\frac{2 \bo{\mu}}{\bo{\sigma}^2}} = \Lpq - \bo{C}_{q}, \quad
%\label{LambdaPlusGamma times Cgamma}
\bo{C}_{q}\Lpq = \Delta_{\frac{2}{\bo{\sigma}^{2}}}\Bigl[-\bo{Q} + q \textbf{I}\Bigr]
\end{equation}
and
\begin{equation}
\label{Relation D Lm}
\Delta_{\frac{2 \bo{\mu}}{\bo{\sigma}^2}} = \bo{D}_{q} - \Lm_{q}, \quad \bo{D}_{q}\Lm_{q} = \Delta_{\frac{2}{\sigma^2}}\Bigl[-\bo{Q} + q \textbf{I}\Bigr],
\end{equation}
where
\[ \bo{C}_{q} = (\Lpq+\Lmq)\Lmq(\Lpq+\Lmq)^{-1},\quad
\bo{D}_{q} = \Bigl(\Lp_{q} + \Lm_{q}\Bigr)\Lp_{q}\Bigl(\Lp_{q} + \Lm_{q}\Bigr)^{-1}.
\]
\end{lemma}
\begin{proof}
\\
Using equations  $(\ref{Characteristic equation})$ altogether, one can obtain
\begin{equation}
\nonumber
 %&\Delta_{\frac{\bo{\sigma}^2}{2}}(\Lpq)^{2} - \Delta_{\bo{\mu}}\Lpq + \Bigl(\bo{Q} - q \textbf{I}\Bigr) = \Delta_{\frac{\bo{\sigma}^2}{2}}(\Lmq)^{2} +
%\Delta_{\bo{\mu}}\Lmq + \Bigl(\bo{Q} - q \textbf{I}\Bigr),\\
\Delta_{\frac{\bo{\sigma}^2}{2}}\Bigl( (\Lpq)^{2} - (\Lmq)^{2}\Bigr) =
\Delta_{\bo{\mu}}\Bigl(\Lpq+\Lmq \Bigr),
\end{equation}
hence,
\begin{equation}
\nonumber
\begin{split}
\Delta_{\frac{2 \bo{\mu}}{\bo{\sigma}^{2}}} &= \Bigl( (\Lpq)^{2} - (\Lmq)^{2}\Bigr)\Bigl(\Lpq+\Lmq\Bigr)^{-1}\\&
= \Bigl( \Lpq(\Lpq+\Lmq)-(\Lpq+\Lmq)\Lmq\Bigr)\Bigr(\Lpq + \Lmq\Bigr)^{-1}\\
&= \Lpq - \bo{C}_{q}.
\end{split}
\end{equation}

Now, the above relationship together with $($\ref{Characteristic equation}$)$ gives that:
\begin{equation}
\nonumber
\bo{C}_{q}\Lpq = \Delta_{\frac{2}{\bo{\sigma}^{2}}}\Bigl[-\bo{Q} + q \textbf{I}\Bigr].
\end{equation}
The remaining part of the proof can be done in a similar way by using
\[
(\Lpq)^2 - (\Lmq)^2 = (\Lpq+\Lmq)\Lpq - \Lmq(\Lpq+\Lmq).
\]
%instead of
%\[
%(\Lpq)^2 - (\Lmq)^2 = \Lpq(\Lpq+\Lmq) - (\Lpq+\Lmq)\Lmq
%\]
\end{proof}
\\
In the special case of $q = 0$ we will write $\Lp$, $\Lm$, $\bo{C}$ and $\bo{D}$ for $\bo{\Lambda}^{+}_{0}$, $\bo{\Lambda}^{-}_{0}$, $\bo{C}_{0}$ and $\bo{D}_{0}$, respectively.

Note that if $(X,J)$ is a MMBM with one single state $($i.e., one dimensional Brownian motion$)$, we have, for $q\ge 0$,  %and $\rho_1$, $\rho_2$
\[
\mathbf{\Lambda}_q^+= -\rho_2, \quad \mathbf{\Lambda}_q^-=-\rho_1,
\]
where $\rho_1-\rho_2=\frac{2\mu}{\sigma^2}$ and $\rho_1 + \rho_2 = \frac{2\sqrt{\mu^2+2q\sigma^2}}{\sigma^2}$.
In general, for the MMBM, we can only calculate explicit analytical formulas for $\bo{W}^{(q)}(x)$, $\Lp_{q},$ and $\Lm_q$ for some special cases.
 For instance, consider the following parameters
\begin{equation}
\nonumber
\Delta_{\bo{\sigma}} = \left( \begin{array}{cc} \sigma_{1} & 0\\ 0 & \sigma_{2} \end{array}\right), \quad
\Delta_{\bo{\mu}} = \left(\begin{array}{cc} 0&0 \\0&0 \end{array}\right), \quad \bo{Q} = \left(\begin{array}{cc} -q_{11}& q_{11} \\ q_{22}& -q_{22}\end{array}\right) \quad\mbox{and} \quad q > 0,
\end{equation}
for $\sigma_{1}$, $\sigma_{2}$, $q_{11}$, $q_{22} \in \mathbb{R}_{+} $. Then the matrix $\bo{F}(s) - q\textbf{I}$ is of the form
\begin{equation}
\nonumber
\bo{F}(s)-q\textbf{I} = \left(\begin{array}{cc} \frac{\sigma_{1}^2}{2}s^2 - q_{11} - q & q_{11} \\ q_{22} & \frac{\sigma_2^2}{2}s^2 -q_{22} -q\end{array}\right).
\end{equation}
Thus,
\begin{equation}
\nonumber
(\bo{F}(s)-q\textbf{I})^{-1}=\frac{1}{(\frac{\sigma_{1}^2}{2}s^2 -q_{11}-q)(\frac{\sigma_{2}^2}{2}s^2 -q_{22} -q) - q_{11} q_{22}}
\left(\begin{array}{cc} \frac{\sigma_{2}^2}{2}s^2 - q_{22} - q & -q_{11}\\ -q_{22} & \frac{\sigma_{1}^2}{2}s^2 - q_{11} - q\end{array} \right).
\end{equation}
%From this formula, decomposition of fractions and inverse of Laplace transform one can obtain
Inversion of the Laplace transform  (\ref{matrix_W_def}) with respect to $s$ gives:
\begin{align}\label{MMBM scale matrix analytical mu = 0}
\nonumber
\bo{W}^{(q)}(x) =&\left(\begin{array}{cc} 2(q_{22}+q)-\alpha_2^2 \sigma_2^2 & 2 q_{11}
\\[10pt] 2 q_{22} & 2(q_{11}+q)-\alpha_2^2 \sigma_1^2 \end{array}\right)\frac{e^{\alpha_2 x}-e^{-\alpha_2 x}}
{(\alpha_1^2-\alpha_2^2)\alpha_2\sigma_1^2\sigma_2^2} \\[10pt]& -\left(\begin{array}{cc} 2(q_{22}+q)-\alpha_{1}^2 \sigma_2^2 &  2 q_{11} \\[10pt] 2 q_{22} &2(q_{11}+q)-\alpha_1^2 \sigma_1^2 \end{array}\right)\frac{e^{\alpha_{1} x} - e^{- \alpha_{1} x}}{(\alpha_1^2-\alpha_2^2)\alpha_1 \sigma_1^2 \sigma_2^2},
\end{align}
where
\begin{align*}
\nonumber
&\alpha_1 = \frac{\sqrt{M_{q}+\sqrt{(M_q)^2 - 4\sigma_1^2\sigma_2^2K_q}}}{\sigma_1 \sigma_2},
\quad \alpha_2 = \frac{\sqrt{M_q-\sqrt{(M_q)^2 - 4\sigma_1^2\sigma_2^2K_q}}}{\sigma_1 \sigma_2},\\&
 M_{q}=\sigma_1^2(q_{22}+q)+\sigma_2^2(q_{11}+q), \quad K_{q}=(q_{11}+q_{22}+q)q.
\end{align*}
It is straightforward that
\begin{align*}
%\label{MMBM derivative of scale matrix analytical mu = 0}
\nonumber
\bo{W}^{(q)\prime}(x) =&\left(\begin{array}{cc} 2(q_{22}+q)-\alpha_2^2 \sigma_2^2 & 2 q_{11} \\[10pt] 2 q_{22} & 2(q_{11}+q)-\alpha_2^2 \sigma_1^2
\end{array}\right)\frac{e^{\alpha_2 x}+e^{-\alpha_2 x}}{(\alpha_1^2-\alpha_2^2)\sigma_1^2\sigma_2^2} \\[10pt]&- \left(\begin{array}{cc} 2(q_{22}+q)-\alpha_{1}^2 \sigma_2^2 &  2 q_{11} \\[10pt] 2 q_{22} &2(q_{11}+q)-\alpha_1^2 \sigma_1^2 \end{array}\right)\frac{e^{\alpha_{1} x} + e^{- \alpha_{1} x}}{(\alpha_1^2-\alpha_2^2) \sigma_1^2 \sigma_2^2}.
\end{align*}
Our last step is to derive the formulas for $\Lpq$ and $\Lmq$. First, note that $\Lpq = \Lmq$ due to the assumption of $\mu_{1} = \mu_{2} = 0$ and equation $(\ref{Characteristic equation})$. Then  $(\ref{Relation C Lp})$ becomes
\begin{equation}
\nonumber
(\Lpq)^2 = \Delta_{\frac{2}{\bo{\sigma}^2}}\Bigl[-\bo{Q} + q\textbf{I}\Bigr].
\end{equation}
Since $-\alpha_1$ and $-\alpha_2$ are eigenvalues of $\Lpq$, thus after some basic algebra, we get that
\begin{equation}
\nonumber
\Lpq = \Lmq = \left( \begin{array}{cc} \frac{-\sqrt{2\sigma_2^2(\alpha_1+\alpha_2)^2(q_{11}+q)-4q_{11}q_{22}}}{\sigma_1 \sigma_2} &
 \frac{2 q_{11}}{\sigma_1^2} \\[10pt] \frac{2 q_{22}}{\sigma_2^2} & \frac{-\sqrt{2\sigma_1^2 (\alpha_1+\alpha_2)^2(q_{22}+q)- 4 q_{11}q_{22}}}{\sigma_1 \sigma_2 }
\end{array}\right) \frac{1}{\alpha_1+\alpha_2}.
\end{equation}
%To close this subsection we provide an image of the scale matrix.
Finally, we will provide a graphical example of the scale matrix. Consider the following setting of the parameters
\begin{equation}
\nonumber
\Delta_{\bo{\sigma}} = \left(\begin{array}{cc} 1 & 0 \\ 0 & 1.2 \end{array}\right), \quad \Delta_{\bo{\mu}} = \left(\begin{array}{cc} 0& 0 \\ 0&0 \end{array}\right), \quad \bo{Q} = \left(\begin{array}{cc} -0.05 & 0.05\\ 0.1 & -0.1 \end{array} \right), \quad\mbox{and}\quad  q = 0.05.
\end{equation}

%The shape of the first derivative of scale functions plays crucial role in the solving of optimal dividend problems. Therefore, in the below picture we provide also image of the first derivative of scale matrix $($we used $\ref{MMBM derivative of scale matrix analytical mu = 0})$
\begin{figure}[H]
\centering
 \includegraphics[width=16cm]{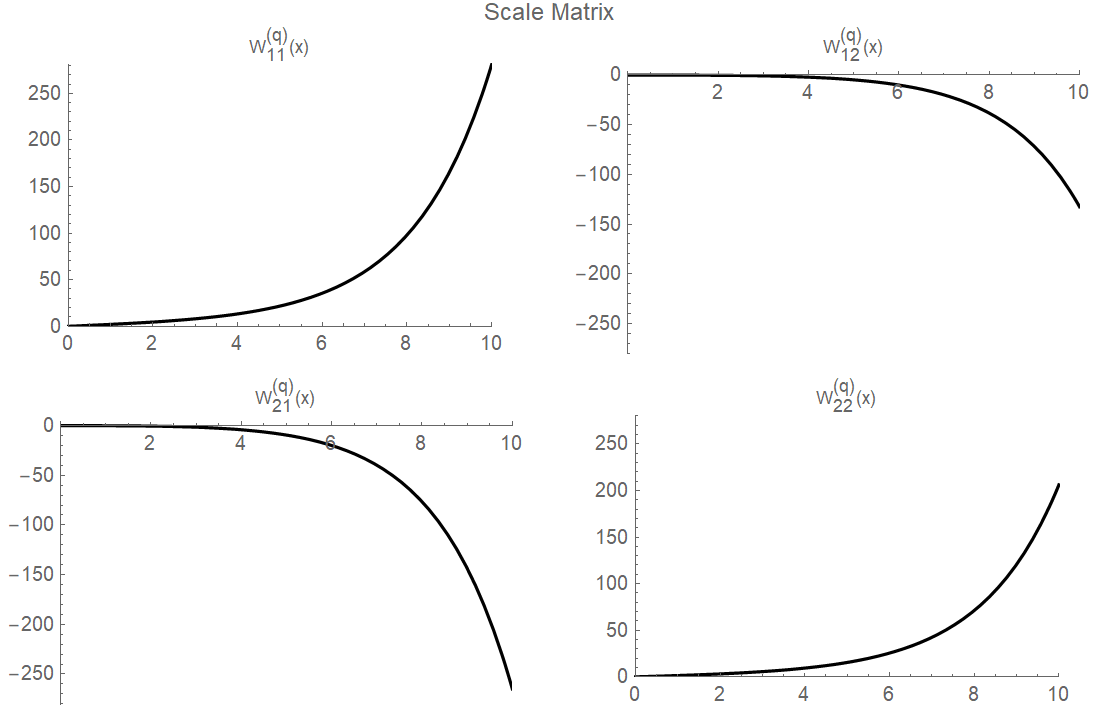}
 \caption{Entries of scale matrix function $\bo{W}^{(q)}$} \label{fig1}
 \end{figure}
Using the formula $($\ref{MMBM scale matrix analytical mu = 0}$)$, the scale matrix $\bo{W}^{(q)}$ is plotted in Figure \ref{fig1}. We can see that the diagonal cells of this matrix have the same shape as the one dimensional scale functions, where off-diagonal ones are reflected in shape. In the subsequent examples, we will provide plots of omega-matrices to compare them to these traditional ones.

\subsection{Constant state-dependent discount rates}
Consider the special case where $\omega_i(x)\equiv \omega_i$ is a constant for all $x \in \mathbb{R}$ and $i \in E$. Therefore, the discounting structure depends on the state of
the chain $J$ only. Before calculating $\omega$-scale matrix, let us state the following proposition.
%{\\ \color{blue}
%Czy nie powinnismy to proposition przeniesc wczesniej? Gdyz jest to ogolny wynik
%}
\begin{proposition}
Let $\omega_i(x)\equiv \omega_i$ for all $x \in \mathbb{R}$ and $i \in E$. The $\omega$-scale matrix has the Laplace transform
\[
\widetilde{\mathcal{W}}^{(\omega)}(s)=(\mathbf{F}(s)-\boldsymbol{\omega})^{-1}.
\]
\end{proposition}

\begin{proof}
Taking the Laplace transform on both sides of (\ref{W}), we have
\[
\widetilde{\mathcal{W}}^{(\omega)}(s)=\widetilde{\mathbf{W}}(s)+\widetilde{\mathbf{W}}(s)\boldsymbol{\omega}\widetilde{\mathcal{W}}^{(\omega)}(s),
\]
which gives
\[
\widetilde{\mathcal{W}}^{(\omega)}(s)=\left(\mathbf{I}- \widetilde{\mathbf{W}}(s)\boldsymbol{\omega}\right)^{-1} \widetilde{\mathbf{W}}(s)=(\mathbf{F}(s)-\boldsymbol{\omega})^{-1}.
\]
\end{proof}

As a example of such $\omega$-scale matrix, we take again the model of Markov modulated Brownian motion with the following parameters: $ \omega_{1}(x) = \omega_{1}$, $\omega_{2}(x) = \omega_{2}$,
\begin{equation}
\nonumber
\Delta_{\bo{\sigma}} = \left(\begin{array}{cc} \sigma_1 & 0 \\ 0 & \sigma_2 \end{array}\right), \quad \Delta_{\bo{\mu}} =
\left(\begin{array}{cc} 0 & 0 \\ 0 & 0 \end{array}\right), \quad\mbox{and}\quad \bo{Q} = \left(\begin{array}{cc} -q_{11} & q_{11} \\ q_{22} & -q_{22}
\end{array} \right).
\end{equation}
Using the same method as in the previous subsection, we will obtain analytical formula for the $\omega$-killed matrix.
Taking the inverse of $\bo{F}(s) - \boldsymbol{\omega}$, one has
\begin{equation}
\nonumber
(\bo{F}(s) - \boldsymbol{\omega})^{-1} = \frac{1}{(\frac{\sigma_1^2}{2}s^2 - q_{11}-\omega_{1})(\frac{\sigma_2^2}{2}s^2 - q_{22} - \omega_{2})-q_{11}q_{22}}
\left(\begin{array}{cc} \frac{\sigma_2^2}{2}s^2 - q_{22}-\omega_{2} & -q_{11}\\ -q_{22} & \frac{\sigma_1^2}{2}s^2-q_{11}-\omega_{1} \end{array}\right),
\end{equation}
%Cause of similarity of the calculations we give only the result
whose Laplace inversion gives
\begin{align*}
\WW^{(\omega)}(x) =& \left(\begin{array}{cc} 2(q_{22}+\omega_2)-\alpha_2^2\sigma_2^2 & 2q_{11} \\ 2q_{22} & 2(q_{11}+\omega_1)
 - \alpha_2^2\sigma_1^2 \end{array}\right)\frac{e^{\alpha_2 x}-e^{-\alpha_2 x}}{(\alpha_1^2-\alpha_2^2)\alpha_2 \sigma_1^2 \sigma_2^2} \\[10pt]&- \left(\begin{array}{cc} 2(q_{22}+q_2)-\alpha_1^2\sigma_2^2 & 2q_{11} \\ 2q_{22} & 2(q_{11}+\delta_1) - \alpha_1^2 \sigma_1^2\end{array}\right) \frac{e^{\alpha_1 x } - e^{-\alpha_1 x}}{(\alpha_1^2-\alpha_2^2)\alpha_1 \sigma_1^2 \sigma_2^2},
\end{align*}
where
\begin{equation}
\nonumber
\begin{split}
& \alpha_1 = \frac{\sqrt{M_{\omega} + \sqrt{(M_{\omega})^2- 4\sigma_1^2\sigma_2^2 K_{\omega}}}}{\sigma_1 \sigma_2}, \quad
 \alpha_2 = \frac{\sqrt{M_{\omega} - \sqrt{(M_{\omega})^2- 4\sigma_1^2\sigma_2^2 K_{\omega}}}}{\sigma_1 \sigma_2},\\
& M_{\omega}=\sigma_1^2(q_{22}+\omega_2) + \sigma_2^2(q_{11}+\omega_1), \quad K_{\omega}=q_{11}\omega_2+\omega_1q_{22} + \omega_1\omega_2.
\end{split}
\end{equation}
Note that, for $\omega_1=\omega_2=q$, the result is consistent with the previous result for the ($q$)-scale matrix $\bo{W}^{(q)}$ in (\ref{MMBM scale matrix analytical mu = 0}). Now, consider the following setting of the parameters
\begin{equation}
\nonumber
\Delta_{\bo{\sigma}} = \left(\begin{array}{cc} 1 & 0 \\ 0 & 1.2 \end{array} \right),\quad
\Delta_{\bo{\mu}} = \left(\begin{array}{cc} 0 & 0 \\ 0 & 0 \end{array}\right),\quad
\bo{Q} = \left(\begin{array}{cc} -0.05 & 0.05 \\ 0.1 & -0.1 \end{array}\right), \quad \omega_{1}(x) = 0.05, \quad \omega_{2}(x) = 0.25,
\end{equation}
which results in the plots of $\omega$-scale matrix in Figure \ref{Fig2}. In Figure \ref{Fig2} one can observe that $\omega$-scale matrix has similar shape as $\bo{W}^{(q)}$.

\begin{figure}[H]
\centering
\includegraphics[width=16cm]{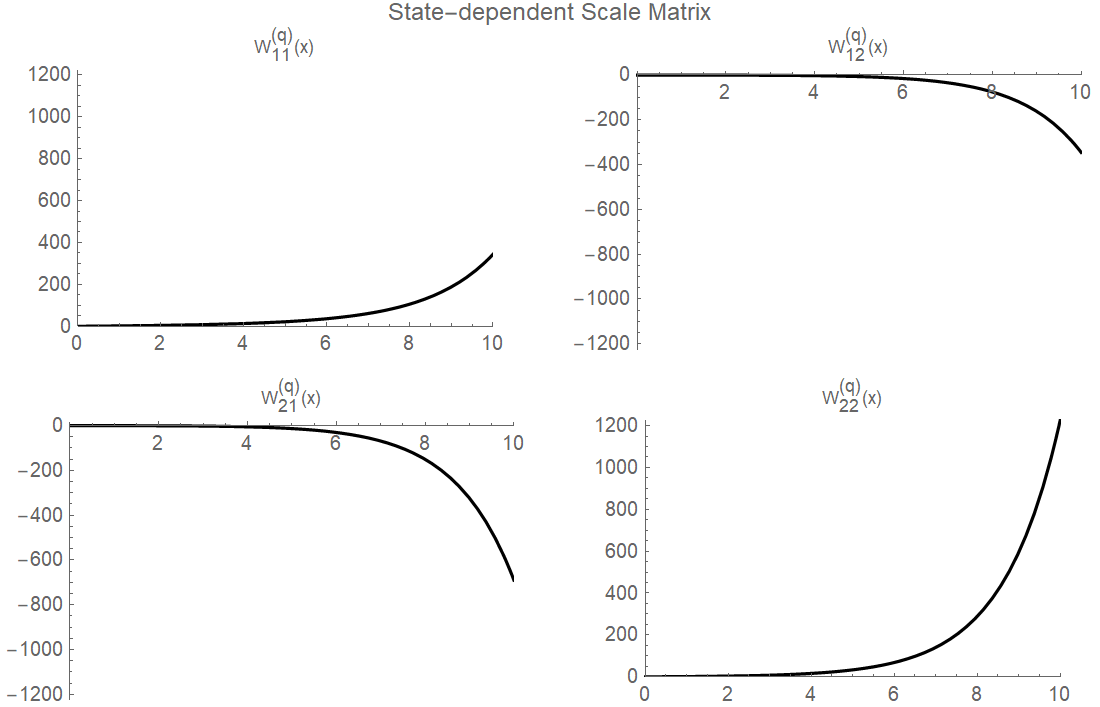}
\caption{Entries of $\omega$-scale matrix function $\WWomega$ for state-dependent $\omega$ function}\label{Fig2}
\end{figure}

\subsection{Step $\omega$-scale matrix}
In this example, we consider omega function as a positive step function which depends only on the position of the process $X$.
Such an assumption is motivated by the situation where the company has the discounting structure depending on its current financial status.
Li and Palmowski \cite{LP16} showed that, in the case of spectrally negative L\'evy processes, such $\omega$-scale functions
have recurrent nature. Same observation holds true for MAPs.

\begin{proposition}\label{Prop_step_omega}
Assume that omega function is of the form
\begin{equation}
\nonumber
 \omega(i,x):=\omega(x)=p_0+\sum_{j=1}^n(p_j-p_{j-1})1_{\{x>x_j\}}, \quad \text{ for all i} \in E,
\end{equation}
where $n \in \mathbb{N}$, $\{p_j\}_{j=0}^n$ is a fixed sequence and $\{x_j\}_{j=1}^n$ is an increasing sequence dividing $\mathbb{R}$ into $(n+1)$ parts. Then the omega matrix $\mathcal{W}^{(\omega)}(x,y)$  satisfies
\[
\mathcal{W}^{(\omega)}(x,y)=\mathcal{W}_n^{(\omega)}(x,y),
\]
for $x>y$, where $\mathcal{W}_n^{(\omega)}(x,y)$ is defined recursively as follows:
\[
\mathcal{W}_0^{(\omega)}(x,y)=\mathbf{W}^{(p_0)}(x-y),
\]
and
\[
\mathcal{W}_{k+1}^{(\omega)}(x,y)=\mathcal{W}_{k}^{(\omega)}(x,y)+(p_{k+1}-p_{k})\int_{x_{k+1}}^{x} \mathbf{W}^{(p_{k+1})}(x-z)\mathcal{W}_k^{(\omega)}(z,y)dz,
\]
for $x>x_{k+1}$ and $k=0, 1, \ldots, n-1$.
\end{proposition}
\begin{proof}

Denote $\omega^{(k)}(x):=p_0+\sum_{j=1}^k(p_j-p_{j-1})1_{\{x>x_j\}}$ with $\omega^{(0)}(x)=p_0$. From Equation (\ref{Wxyd}), we get that
\begin{align}
\mathcal{W}^{(\omega)}_{k}(x,y)&=\mathbf{W}^{(p_{k+1})}(x-y)+\int_{y}^{x} (\omega^{(k)}(z)-p_{k+1})\mathbf{W}^{(p_{k+1})}(x-z)\mathcal{W}_k^{(\omega)}(z,y)dz, \label{Wk} \\
\mathcal{W}^{(\omega)}_{k+1}(x,y)&=\mathbf{W}^{(p_{k+1})}(x-y)+\int_{y}^{x} (\omega^{(k+1)}(z)-p_{k+1})\mathbf{W}^{(p_{k+1})}(x-z)\mathcal{W}_{k+1}^{(\omega)}(z,y)dz. \label{Wk+1}
\end{align}
Note that $\omega^{(k+1)}(z)-p_{k+1}=0$ for $z>{x_{k+1}}$ and $\omega^{(k+1)}(z)=\omega^{(k)}(z)$ for $z\le x_{k+1}$. Thus from Lemma \ref{lem unique}, we have
\[
\mathcal{W}^{(\omega)}_{k+1}(x,y)=\mathcal{W}^{(\omega)}_{k}(x,y),
\]
for $x\le x_{k+1}$. Equation (\ref{Wk+1}) could be rewritten as
\begin{align*}
\mathcal{W}^{(\omega)}_{k+1}(x,y)%&=\mathbf{W}^{(p_{k+1})}(x-y)+\int_{y}^{x} (\omega^{(k+1)}(z)-p_{k+1})\mathbf{W}^{(p_{k+1})}(x-z)\mathcal{W}_{k+1}^{(\omega)}(z,y)dz\\
&=\mathbf{W}^{(p_{k+1})}(x-y)+\int_{y}^{x_k} (\omega^{(k+1)}(z)-p_{k+1})\mathbf{W}^{(p_{k+1})}(x-z)\mathcal{W}_{k+1}^{(\omega)}(z,y)dz\\
&=\mathbf{W}^{(p_{k+1})}(x-y)+\int_{y}^{x_k} (\omega^{(k)}(z)-p_{k+1})\mathbf{W}^{(p_{k+1})}(x-z)\mathcal{W}_{k+1}^{(\omega)}(z,y)dz\\
&=\mathcal{W}^{(\omega)}_{k}(x,y)-\int_{x_k}^{x} (\omega^{(k)}(z)-p_{k+1})\mathbf{W}^{(p_{k+1})}(x-z)\mathcal{W}_k^{(\omega)}(z,y)dz,
\end{align*}
where the last step uses (\ref{Wk}). The proof is completed by noticing that $\omega^{(k)}(z)-p_{k+1}=p_{k}-p_{k+1}$ for $z>x_{k+1}$.
\end{proof}

Note also that the similar considerations will lead to the same result for the second $\omega$-scale matrix $\mathcal{Z}^{(\omega)}$.

In the next proposition, we will compute the matrix $\mathcal{W}^{(\omega)}$ for one particular case.
\begin{proposition}
Let $(X,J)$ be a Markov modulated Brownian motion with $\mu_i \in \mathbb{R}$ and $\sigma_{i}^2 > 0$ for all $i \in E$.
Assume that ($n=1$) $\lbrace p_{j} \rbrace_{j=0}^{n} = \lbrace p_{0},p_{1}\rbrace$ and $\lbrace x_{j} \rbrace_{j = 1}^{n} = \lbrace x_{1} \rbrace$ with $p_0,p_1,x_1$ being positive numbers. Then for $x \leq x_1$,
\begin{equation}
\nonumber
\WWomega (x,y) = \bo{W}^{(p_0)}(x-y),
\end{equation}
and for $x > x_1$,
\begin{align*}
\nonumber
\WWomega (x,y) =& \Bigl(e^{-\Lp_{p_{1}}(x-x_{1})}\Bigl(\Lp_{p_{1}}+\Lm_{p_{1}}\Bigr)^{-1}\Lm_{p_{1}} + e^{\Lm_{p_{1}}(x-x_{1})}\Bigl(\Lp_{p_{1}}+\Lm_{p_{1}}\Bigr)^{-1} \Lp_{p_{1}}\Bigr)\bo{W}^{(p_{0})}(x_{1}-y) \\&- \bo{W}^{(p_{1})}(x-x_{1})\Delta_{\frac{\sigma^2}{2}}\bo{W}^{(p_{0})'}(x_{1}-y).
\end{align*}
\end{proposition}
\begin{proof}
Note that the case when $x \leq x_{1}$ is a straightforward conclusion from Proposition \ref{Prop_step_omega}. For $x > x_{1}$, from previous proposition and (\ref{MMBM Scale matrix formula}), we have
\begin{align}
\mathcal{W}_{1}(x,y) =& \mathcal{W}_{0}(x,y) + (p_{1}-p_{0})\int_{x_{1}}^{x} \bo{W}^{(p_{1})}(x-z)\mathcal{W}_{0}(z,y)dz\nonumber
\\=&  \bo{W}^{(p_{0})}(x-y) + (p_{1}-p_{0})\int_{x_{1}}^{x}\Bigl(e^{-\Lp_{p_{1}} (x-z)}\Xi_{p_{1}}e^{-\Lp_{p_{0}}(z-y)} \label{Step omega-scale W1 integral formula}
\\ &- e^{-\Lp_{p_{1}}(x-z)}\Xi_{p_{1}}e^{\Lm_{p_{0}} (z-y)} - e^{\Lm_{p_{1}}(x-z)}\Xi_{p_{1}}e^{-\Lp_{p_{0}}(z-y)} + e^{\Lm_{p_{1}}(x-z)}\Xi_{p_{1}}e^{\Lm_{p_{0}}(z-y)}\Bigr)dz \;\Xi_{p_{0}}.\nonumber
\end{align}
We start from identifying the following integral appearing in Equation \eqref{Step omega-scale W1 integral formula}:
\begin{equation}\label{Step omega : M integral}
\int_{x_{1}}^{x} \Bigl(e^{-\Lp_{p_{1}} (x-z)}\Xi_{p_{1}}e^{-\Lp_{p_{0}}(z-y)}\Bigr)dz.
\end{equation}
Consider $($\ref{Step omega : M integral}$)$ as a function $M_{1}:A \rightarrow \mathbb{R}^{N \times N}$, where
\begin{equation}
\nonumber
A = \lbrace (x,y) : x \geq x_{1}, x> y\rbrace,
%A = \lbrace (x,y) : x \geq x_{1}  \geq y , x \neq y\rbrace
\end{equation} and $N$ is the size of the matrix $\bo{W}^{(p_{0})}$. Then
\begin{equation}
\nonumber
M_{1}(x,y) = \int_{x_{1}}^{x} \Bigl(e^{-\Lp_{p_{1}} (x-z)}\Xi_{p_{1}}e^{-\Lp_{p_{0}}(z-y)}\Bigr)dz=e^{-\Lp_{p_{1}} x}\int_{x_{1}}^{x} \Bigl(e^{\Lp_{p_{1}} z}\Xi_{p_{1}}e^{-\Lp_{p_{0}}z}\Bigr)dz \ e^{\Lp_{p_{0}}y}.
\end{equation}
Let \begin{equation}
\nonumber
\quad K_{1}(x):=M_{1}(x,x_{1}) =   \int_{x_{1}}^{x} \Bigl(e^{-\Lp_{p_{1}} (x-z)}\Xi_{p_{1}}e^{-\Lp_{p_{0}}(z-x_{1})}\Bigr)dz.
\end{equation}
%By taking partial derivatives of $M_1$ with respect to $x$ and $y$ we get
%\begin{equation}
%\nonumber
%\begin{cases} \frac{\partial M_{1}(x,y)}{\partial x} = -\Lp_{p_{1}} M_{1}(x,y) + \Xi_{p_{1}}e^{-\Lp_{p_{0}}(x-y)}, \\ \frac{\partial M_{1}(x,y)}{\partial y} = M_{1}(x,y) \Lp_{p_{0}} \end{cases}
%\end{equation}
%with the boundary conditions:
%\begin{equation}
%\nonumber
%M_{1}(x_{1},y) = \bo{0} \quad \text{and} \quad K_{1}(x):=M_{1}(x,x_{1}) =   \int_{x_{1}}^{x} \Bigl(e^{-\Lp_{p_{1}} (x-z)}\Xi_{p_{1}}e^{-\Lp_{p_{0}}(z-x_{1})}\Bigr)dz.
%\end{equation}
The derivative of $K_{1}(x)$ equals
\begin{equation}\label{K prime}
K^{\prime}_{1}(x) = -\Lp_{p_{1}} K_{1}(x) + \Xi_{p_{1}}e^{-\Lp_{p_{0}} (x-x_{1})} \quad \textrm{with the boundary condition}
\quad K_{1}(x_{1}) = \bo{0}.
\end{equation}
We will prove that the solution of above differential equation is of the form
\begin{equation}\label{Hyphothesis for K}
K_{1}(x) = \bo{C}e^{-\Lp_{p_0}(x-x_{1})} - e^{-\Lp_{p_1}(x-x_{1})}\bo{C},
\end{equation}
where $\bo{C}$ is some constant matrix. To do this, we need to put our proposition for $K_1(x)$ into $(\ref{K prime})$ and after some calculation we get that (\ref{Hyphothesis for K}) is indeed our solution if the following equation holds true:
\begin{equation}\label{Sylvester eq C}
\Lp_{p_{1}}\bo{C} - \bo{C} \Lp_{p_{0}} =\Xi_{p_{1}}.
\end{equation}
The above equality is an example of well known Sylvester equation. Usually to solve equations of this type we must use numerical methods, however in this case we can guess the formula for $\bo{C}$:
\begin{equation}
\nonumber
\bo{C} = -\Bigl(\Lp_{p_{1}}+\Lm_{p_{1}}\Bigr)^{-1} \Bigl(\Lp_{p_{0}}+\Lm_{p_{1}}\Bigr)\cdot \frac{1}{p_{1}-p_{0}}.
\end{equation}
We need to check if such formula for $\bo{C}$ is indeed correct. Therefore, from equation (\ref{Sylvester eq C}) one can get that
\begin{equation}
\nonumber
\begin{split}
&-\Lp_{p_{1}} \Bigl(\Lp_{p_{1}}+\Lm_{p_{1}}\Bigr)^{-1} \Bigl(\Lp_{p_{0}}+\Lm_{p_{1}}\Bigr)\cdot \frac{1}{p_{1}-p_{0}} + \Bigl(\Lp_{p_{1}}+\Lm_{p_{1}}\Bigr)^{-1} \Bigl(\Lp_{p_{0}}+\Lm_{p_{1}}\Bigr)\cdot \frac{1}{p_{1}-p_{0}} \Lp_{p_{0}} = \Xi_{p_{1}}\\
&\Bigl[\Bigl(\Lp_{p_{1}} + \Lm_{p_{1}}\Bigr)\Lp_{p_{1}}\Bigl(\Lp_{p_{1}} + \Lm_{p_{1}}\Bigr)^{-1} \Bigl(\Lp_{p_{0}}+\Lm_{p_{1}}\Bigr) - \Bigl(\Lp_{p_{0}} + \Lm_{p_{1}}\Bigr)\Lp_{p_{0}}\Bigr]\cdot\frac{1}{p_{1}-p_{0}} = \Delta_{\frac{2}{\sigma^2}} \\
&\Bigl[\Bigl(\Delta_{\frac{2\mu}{\sigma^2}} + \Lm_{p_{1}}\Bigr)\Bigl(\Lp_{p_{0}} + \Lm_{p_{1}}\Bigr) - (\Lp_{p_{0}})^2 - \Lm_{p_{1}}\Lp_{p_{0}}\Bigr]\cdot \frac{1}{p_{1}-p_{0}} = \Delta_{\frac{2}{\sigma^2}}\\
&\Bigl[\Delta_{\frac{2\mu}{\sigma^2}}\Lp_{p_{0}} + \Delta_{\frac{2\mu}{\sigma^2}}\Lm_{p_{1}} + (\Lm_{p_{1}})^2 - (\Lp_{p_{0}})^2\Bigr]\cdot \frac{1}{p_{1}-p_{0}} = \Delta_{\frac{2}{\sigma^2}}\\
&\Bigl[\Bigl(\bo{Q}-p_{0}\textbf{I}\Bigr)
-\Bigl(\bo{Q}-p_{1}\textbf{I}\Bigr)\Bigr]\cdot \frac{1}{p_{1}-p_{0}} = \textbf{I}\\
&\textbf{I} = \textbf{I}.
\end{split}
\end{equation}
In the second line of above calculations we used the definition of $\Xi_{p_1}$. Third equality follows from the second by the relation (\ref{Relation D Lm}). Finally, to get the fifth equation we used (\ref{Relation C Lp}) and again (\ref{Relation D Lm}).\\
Therefore, $K_1(x)$ is a solution to differential equation \ref{K prime}.
Returning to $M_{1}(x,y)$ it is now straightforward to guess and check the formula for $M_{1}$, namely
\begin{equation}
\nonumber
M_{1}(x,y) = \bo{C} e^{-\Lp_{p_{0}}(x-y)} - e^{-\Lp_{p_{1}}(x-x_{1})}\bo{C} e^{-\Lp_{p_{0}}(x_{1}-y)}.
\end{equation}
Now, using similar reasoning as for deriving $M_{1}$ we can identify other integrals appearing in Equation \eqref{Step omega-scale W1 integral formula}:
\begin{equation}
\begin{split}
\nonumber
&M_{2}(x,y) = \int_{x_{1}}^{x} e^{-\Lp_{p_{1}}(x-z)}\Xi_{p_{1}}e^{\Lm_{p_{0}}(z-y)}dz, \\
&M_{3}(x,y) = \int_{x_{1}}^{x} e^{\Lm_{p_{1}}(x-z)}\Xi_{p_{1}}e^{-\Lp_{p_{0}}(z-y)}dz, \\
&M_{4}(x,y) = \int_{x_{1}}^{x} e^{\Lm_{p_{1}}(x-z)}\Xi_{p_{1}}e^{\Lm_{p_{0}}(z-y)}dz.
\end{split}
\end{equation}
Precisely,
\begin{equation}
\nonumber
\begin{split}
&M_{1}(x,y) =  \bo{C} e^{-\Lp_{p_{0}}(x-y)} - e^{-\Lp_{p_{1}}(x-x_{1})}\bo{C} e^{-\Lp_{p_{0}}(x_{1}-y)},\\
&M_{2}(x,y) = \bo{D} e^{\Lm_{p_{0}}(x-y)} - e^{-\Lp_{p_{1}}(x-x_{1})}\bo{D}e^{\Lm_{p_{0}}(x_{1}-y)}, \\
& M_{3}(x,y) = \bo{E} e^{-\Lp_{p_{0}}(x-y)} - e^{\Lm_{p_{1}}(x-x_{1})}\bo{E}e^{-\Lp_{p_{0}}(x_{1}-y)}, \\
& M_{4}(x,y)= \bo{F}e^{\Lm_{p_{0}} (x-y)} - e^{\Lm_{p_{1}}(x-x_{1})}\bo{F}e^{\Lm_{p_{0}}(x_{1}-y)},
\end{split}
\end{equation}
where matrices $\bo{C},\bo{D},\bo{E},\bo{F}$ are given by
\begin{equation}
%\label{Formulas for C,D,E,F}
\nonumber
\begin{split}
&\bo{C} = -\Bigl(\Lp_{p_{1}} + \Lm_{p_{1}}\Bigr)^{-1}\Bigl(\Lp_{p_{0}}+\Lm_{p_{1}}\Bigr)\cdot \frac{1}{p_{1}-p_{0}}, \\
&\bo{D} = \Bigl(\Lp_{p_{1}} + \Lm_{p_{1}}\Bigr)^{-1}\Bigl(\Lm_{p_{0}}-\Lm_{p_{1}}\Bigr)\cdot \frac{1}{p_{1}-p_{0}},\\
&\bo{E} = -\Bigl(\Lp_{p_{1}} + \Lm_{p_{1}}\Bigr)^{-1}\Bigl(\Lp_{p_{0}}-\Lp_{p_{1}}\Bigr)\cdot \frac{1}{p_{1}-p_{0}},\\
&\bo{F} = \Bigl(\Lp_{p_{1}} + \Lm_{p_{1}}\Bigr)^{-1}\Bigl(\Lm_{p_{0}}+\Lp_{p_{1}}\Bigr)\cdot \frac{1}{p_{1}-p_{0}}.
\end{split}
\end{equation}

Thus from \eqref{Step omega-scale W1 integral formula}, for $x > x_{1}$,
\begin{align*}
\WW_{1}^{(\omega)}(x,y) =& \Bigl(e^{-\Lp_{p_{0}}(x-y)}-e^{\Lm_{p_{0}}(x-y)}\Bigr)\Xi_{p_{0}}\\
& + (p_{1}-p_{0})\Bigl(M_{1}(x,y)-M_{2}(x,y) -  M_{3}(x,y)
+ M_{4}(x,y)\Bigr) \Xi_{p_{0}} \\
=& \Bigl[\Bigl(\textbf{I} - (p_{1}-p_{0})\Bigl(\bo{E}-\bo{C}\Bigr)\Bigr)e^{-\Lp_{p_{0}}(x-y)} -\Bigl(\textbf{I} - (p_{1}-p_{0})\Bigl(\bo{D}+\bo{F}\Bigr)\Bigr) e^{\Lm_{p_{0}}(x-y)} \\
&+ (p_{1}-p_{0})\Bigl( e^{\Lm_{p_{1}}(x-x_{1})}
\Bigl(\bo{E}e^{-\Lp_{p_{0}}(x_{1}-y)} - \bo{F} e^{\Lm_{p_{0}}(x_{1}-y)}\Bigr) \\
&- e^{-\Lp_{p_{1}}(x-x_{1})}
\Bigl(\bo{C}e^{-\Lp_{p_{0}}(x_{1}-y)} - \bo{D} e^{\Lm_{p_{0}}(x_{1}-y)}\Bigr)\Bigr)\Bigr]\Xi_{p_{0}}\\ =&
\Bigl[e^{-\Lp_{p_{1}}(x-x_{1})}\Bigl(\Lp_{p_{1}}+\Lm_{p_{1}}\Bigr)^{-1}\Lm_{p_{1}}
+ e^{\Lm_{p_{1}}(x-x_{1})}\Bigl(\Lp_{p_{1}}+\Lm_{p_{1}}\Bigr)^{-1} \Lp_{p_{1}}\Bigr]\bo{W}^{(p_{0})}(x_{1}-y)
\\&- \bo{W}^{(p_{1})}(x-x_{1})\Delta_{\frac{\sigma^2}{2}}\bo{W}^{(p_{0})\prime}(x_{1}-y),
\end{align*}
where we notice the facts that
\begin{equation}
%\label{Relation between E,C and D,F}
\nonumber
(p_{1}-p_{0})\Bigl(\bo{E}-\bo{C}\Bigr) = \textbf{I}, \quad (p_{1}-p_{0})\Bigl(\bo{D}+\bo{F}\Bigr) = \textbf{I}.
\end{equation}
This completes the proof of this proposition. Note that the uniqueness of this result is straightforward conclusion from Lemma \ref{lem unique}
\end{proof}
\begin{remark}
\rm In general, if we choose to divide $\mathbb{R}$ into more intervals,
%(i.e. consider $\lbrace p_{j} \rbrace_{j=0}^{n}$ and $\lbrace x_{j} \rbrace_{j = 1}^{n}$ for $n\geq 2$ )
similar idea could be used for the computations of $\omega$-scale matrix.
\end{remark}

We take the following parameters for the numerical analysis:
\begin{align*}
\nonumber
&\Delta_{\bo{\sigma}} = \left(\begin{array}{cc} 0.7 & 0 \\ 0 & 0.85 \end{array}\right), \quad\Delta_{\bo{\mu}} = \left(\begin{array}{cc} 0.1 & 0 \\ 0 & -0.1 \end{array}\right), \quad \bo{Q} = \left(\begin{array}{cc} -0.1 & 0.1 \\ 0.3 & -0.3 \end{array}\right),\\& p_{0} = 0.25,\quad p_{1} = 0.03, \quad x_{1} = 4.
\end{align*}

Note that we do not assume that $\Delta_{\bo{\mu}} \neq \bo{0}$ and thus we cannot use the formula $(\ref{MMBM scale matrix analytical mu = 0})$.
Therefore for the computations, we used numerical package \cite{I11} instead. %$\textsc{Spectrally-negative Markov additive processes}$

From Figure \ref{Fig3} one can see that in every cell we have interesting relation that $\WW^{(\omega)}$ lies between $\bo{W}^{(p_0)}$ and $\bo{W}^{(p_1)}$ and this functions are similar in shape.

\begin{figure}[H]
\centering
\includegraphics[width=16cm]{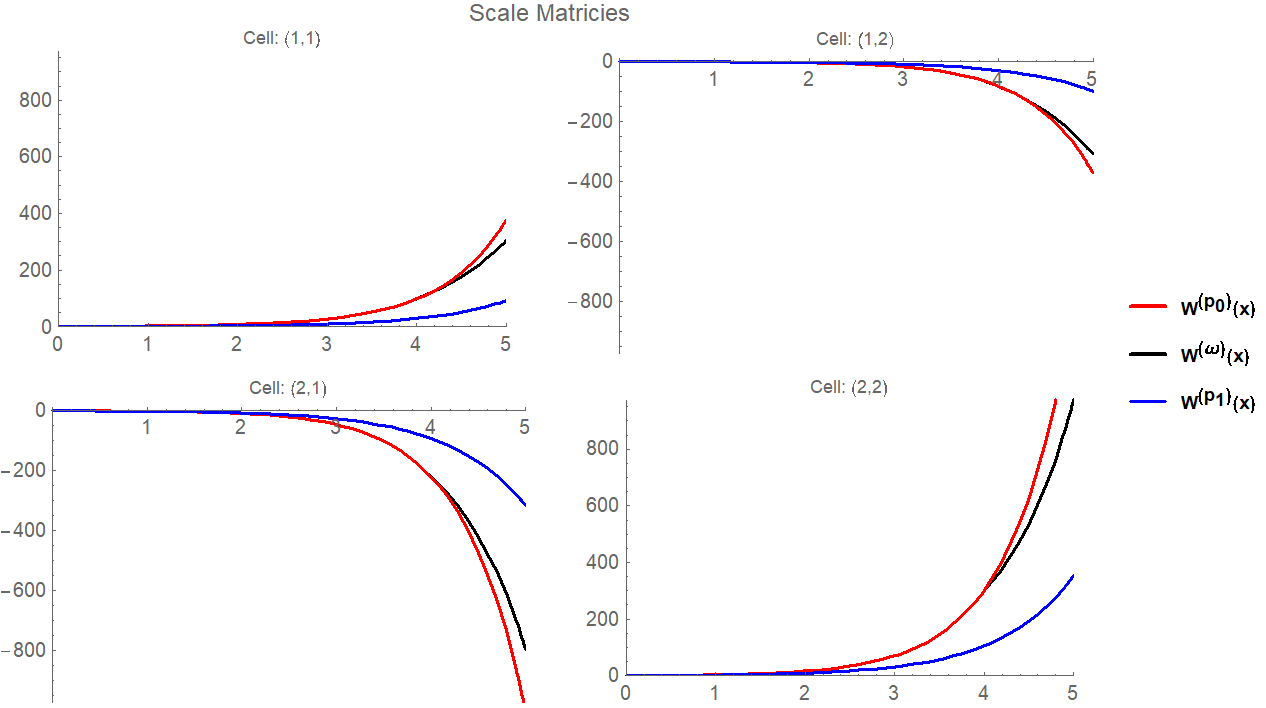}
\caption{Comparison between entries of scale matrix $\bo{W}^{(p_0)}$,$\bo{W}^{(p_1)}$ and entries of $\omega$-scale matrix function $\WWomega$}\label{Fig3}
\end{figure}

\subsection{Omega model}
In Section \ref{Div}, we considered dividend problem in the general Markov additive model and we derived the formula for the value function in the terms of $\omega$-scale matrix.
In this subsection, we will analyze it for the specific choice of $\omega$ function:
\[
\omega(i,x) := \omega(x)=\Bigl (\gamma_0+\gamma_1(x+d) \Bigr)1_{\{-d\le x\le 0\}}, \quad \text{for all } i \in E,
\]
and for MAP being a Markov modulated Brownian motion.
Similar model for the L\'evy-risk process was analyzed in Li and Palmowski \cite{LP16}.

Fix a constant force of interest $\delta \geq 0$. Using (\ref{Woxy}) we obtain that  $\mathcal{W}^{(\omega + \delta)}$ satisfy following equation: for $x\in [-d,0]$,
\begin{align*}
\WW^{(\omega+\delta)}(x,-d)&=\mathbf{W}(x+d)+\int_{-d}^{x} (\omega(z)+\delta)\mathbf{W}(x-z)\mathcal{W}^{(\omega)}(z,-d)dz\\
&=\mathbf{W}(x+d)+\int_{0}^{x+d} (\omega(y-d)+\delta)\mathbf{W}(x+d-y)\mathcal{W}^{(\omega)}(y-d,-d)dy \\
&=\mathbf{W}^{(\gamma_0 +\delta)}(x+d)+ \gamma_1 \int_{0}^{x+d}y\mathbf{W}^{(\gamma_0 + \delta)}(x+d-y)\mathcal{W}^{(\omega)}(y-d,-d)dy. %\text { for } x\in [-d,0].
\end{align*}
%where in the last line we take  .

Now, let $z=x+d\ge 0$, $\omega_1(z) := \omega(x+d) = (\gamma_0 + \gamma_1 z ) \textbf{1}_{\lbrace 0 \leq z \leq d \rbrace}$ and
\begin{align}\label{G eq}
\mathbf{G}(z):=\WW^{(\omega + \delta)}(z-d,-d)=\WW^{(\omega + \delta)}(x,-d),
\end{align}
Then we  can rewrite equation for $\WW^{(\omega + \delta)}$ as
\begin{equation*}
\bo{G}(z) = \bo{W}^{( \gamma_0 + \delta)}(z) + \gamma_1 \int_{0}^{z}  y \bo{W}^{(\gamma_0 + \delta)}(z-y)\bo{G}(y)dy.
\end{equation*}
From equation $(\ref{MMBM Scale matrix formula})$ for $\bo{W}^{( \gamma_0 + \delta)}$ we obtain that
%\[
%\left(\frac{d}{dz} -\mathbf{C}_q\right)\left(\frac{d}{dz}+\mathbf{\Lambda}_q^+\right)\mathbf{W}^{(\gamma_0)}(z)=\mathbf{0},
%\]
%where $\mathbf{C}_q=(\mathbf{\Lambda}_q^++\mathbf{\Lambda}_q^-)\mathbf{\Lambda}_q^-(\mathbf{\Lambda}_q^++\mathbf{\Lambda}_q^-)^{-1}$.
%
%For $z\in[0,d]$ (or $x\in[-d,0]$, we have
%\[
%\left(\frac{d}{dz} -\mathbf{C}_q\right)\left(\frac{d}{dz}+\mathbf{\Lambda}_q^+\right)\mathbf{G}(z)=2\gamma_1 z \Delta_{1/{\sigma}^2} \mathbf{G}(z),
%\]
%with the boundary conditions
%\[
%\mathbf{G}(0)=\mathbf{0}, \mathbf{G}^{\prime}(0)=2 \Delta_{1/{\sigma}^2}.
%\]
\begin{equation}\label{diffScaleZero}
\Bigl(\frac{d}{dz} - \bo{C}_{\gamma_{0}+\delta}\Bigr) \Bigl(\frac{d}{dz} + \Lambda_{\gamma_{0}+\delta}^{+}\Bigr) \bo{W}^{(\gamma_0 + \delta)}(z) = \bo{0},
\end{equation}
where $\bo{C}_{\gamma_{0}+\delta} = (\Lpgd + \Lmgd)\Lmgd(\Lpgd+\Lmgd)^{-1}$.

Starting from $(\ref{G eq})$, for $z \in [0,d]$ $($ or equivalently for $x \in [-d,0])$ we have
\begin{equation}\label{diffMain}
\Bigl(\frac{d}{dz} - \bo{C}_{\gamma_{0}+\delta}\Bigr) \Bigl(\frac{d}{dz} + \Lambda_{\gamma_{0}+\delta}^{+}\Bigr)\bo{G}(z) = \gamma_{1}z\Delta_{\frac{2}{ \bo{\sigma}^{2}}}\bo{G}(z)
\end{equation}
with the boundary conditions
\begin{equation*}
\bo{G}(0) = \bo{0}, \qquad \bo{G^{\prime}}(0) = \Delta_{\frac{2}{\bo{\sigma}^{2}}}.
\end{equation*}
For better understanding of the nature of the above differential matrix equation we rewrite it into the following form
\begin{equation*}\label{Finall Eq}
\bo{G}^{\prime \prime}(z) + \Bigl(\Lpgd - C_{\gamma_{0}+\delta}\Bigr)\bo{G}^{\prime}(z) - \Bigl(\bo{C}_{\gamma_{0}+\delta}\Lpgd + 2\gamma_{1}z\Delta_{1 / \bo{\sigma}^{2}}\Bigr)\bo{G}(z) = \bo{0},
\end{equation*}
which could be simplified to, by $($\ref{Relation C Lp}$)$,\\
\begin{equation*}\label{Finall Eq convinent form}
\Delta_{\frac{\bo{\sigma}^2}{2}}\bo{G}^{\prime \prime}(z) + \Delta_{\bo{\mu}}\bo{G}^{\prime}(z) + \bo{Q}\bo{G}(z) - (\omega_1(z)+\delta)\bo{G}(z) = \bo{0}, \quad \text{for } z \in [0,d].
\end{equation*}

Now, we will treat the case of $z \geq d$ (or equivalently for $x \geq 0$). We first rewrite formula
\begin{equation*}\label{WWomega when z greater than d}
\WW^{(\omega + \delta)}(x;-d) = \bo{W}(x+d) + \int_{0}^{x+d}\omega(y-d)\bo{W}(x+d-y)\WW^{(\omega + \delta)}(y-d;-d)dy, \quad \text{for } x\geq 0,
\end{equation*}
in the terms of matrix $\bo{G}(z)$ with respect to $z \geq d$:
\begin{equation}
\nonumber
\bo{G}(z) = \bo{W}(z) + \int_{0}^{d}(\delta+ (\gamma_0 + \gamma_1 y))\bo{W}(z-y)\bo{G}(y)dy + \delta \int_{d}^{z} \bo{W}(z-y) \bo{G}(y) dy.
\end{equation}
Similar to $($\ref{diffScaleZero}$)$ and $($\ref{diffMain}$)$, we have, respectively
\begin{equation}
\nonumber
\Bigl(\frac{d}{dz}- \bo{C}\Bigr)\Bigl(\frac{d}{dz} + \Lp\Bigr)\bo{W}(z) = \bo{0},
\end{equation}
and %the following differential equation
\begin{equation*}
\Bigl(\frac{d}{dz}- \bo{C}\Bigr)\Bigl(\frac{d}{dz} + \Lp\Bigr)\bo{G}(z) = \delta \Delta_{\frac{2}{\bo{\sigma}^2}}\bo{G}(z), \quad \text{for } z \geq d,
\end{equation*}
where $\bo{C} = (\Lp + \Lm)\Lm(\Lp+\Lm)^{-1}$.
Using $($\ref{Relation C Lp}$)$ for $q = 0$, we get that
\begin{equation*}
\Delta_{\frac{\bo{\sigma}^2}{2}}\bo{G}^{\prime \prime}(z) + \Delta_{\bo{\mu}}\bo{G}^{\prime}(z) + \bo{Q}\bo{G}(z) - \delta\bo{G}(z) = \bo{0}, \quad \text{for } z > d.
\end{equation*}
Summarizing, $\bo{G}(z)$ satisfies the following differential equations:
\begin{equation*} \begin{array}{rl}
 \Delta_{\frac{\bo{\sigma}^2}{2}}\bo{G}^{\prime \prime}(z) + \Delta_{\bo{\mu}}\bo{G}^{\prime}(z) + \bo{Q}\bo{G}(z) - (\omega_1(z)+\delta) \bo{G}(z) = \bo{0},& \textrm{for $z \in [0,d]$}
\\ \Delta_{\frac{\bo{\sigma}^2}{2}}\bo{G}^{\prime \prime}(z) + \Delta_{\bo{\mu}}\bo{G}^{\prime}(z) + \bo{Q}\bo{G}(z) - \delta \bo{G}(z)= \bo{0}, & \textrm{for $z>d$}
\end{array} .
\end{equation*}
%\begin{equation}
%\nonumber
%\begin{split}
%&\Delta_{\frac{\bo{\sigma}^2}{2}}\bo{G}^{\prime \prime}(z) + \Delta_{\bo{\mu}}\bo{G}^{\prime}(z) + \bo{Q}\bo{G}(z) - [\omega_1(z)+\delta] \bo{G}(z) = \bo{0}, \quad \text{for } z \in [0,d], \\
%& \Delta_{\frac{\bo{\sigma}^2}{2}}\bo{G}^{\prime \prime}(z) + \Delta_{\bo{\mu}}\bo{G}^{\prime}(z) + \bo{Q}\bo{G}(z) - \delta \bo{G}(z)= \bo{0} \quad \text{for } z > d,
%\end{split}
%\end{equation}
with the boundary conditions $\bo{G}(0) = \bo{0},$ and $\bo{G}^{\prime}(0) = \Delta_{\frac{2}{\bo{\sigma^2}}}$.

Therefore from \eqref{G eq} for $x \in [-d,0]$ we obtain:
\begin{equation}
\nonumber
\Delta_{\frac{\bo{\sigma}^2}{2}}\WW^{(\omega + \delta)\prime \prime}(x,-d) + \Delta_{\bo{\mu}}\WW^{(\omega +\delta)\prime}(x,-d)
-((\omega(x+d)+\delta)\bo{I}-\bo{Q})\WW^{(\omega + \delta)}(x,-d)=\bo{0},
\end{equation}
and for $x > 0$,
\begin{equation}
\nonumber
\Delta_{\frac{\bo{\sigma}^2}{2}}\WW^{(\omega + \delta)\prime \prime}(x,-d) + \Delta_{\bo{\mu}}\WW^{(\omega +\delta)\prime}(x,-d)
-(\delta\bo{I}-\bo{Q})\WW^{(\omega + \delta)}(x,-d)=\bo{0},
\end{equation}
%\begin{equation}
%\nonumber
%\begin{split}
%&\Delta_{\frac{\bo{\sigma}^2}{2}}\WW^{(\omega+\delta)''}(x;-d) + \Delta_{\bo{\mu}}\WW^{(\omega+\delta)'}(x;-d) + \bo{Q}\WW^{(\omega + \delta)}(x;-d) - \delta \WW^{(\omega + \delta)}(x,-d) = \bo{0},
%\end{split}
%\end{equation}
with the boundary conditions $\WW^{(\omega + \delta)} (-d,-d) = \bo{0}$ and $\WW^{(\omega+\delta)\prime}(-d,-d) = \Delta_{\frac{2}{\bo{\sigma}^{2}}}$.

Before we proceed to the numerical example, we recall that $N$ is the cardinality of the state space $E$ and $\WW^{(\omega + \delta)}$ maps $\mathbb{R}$ into $\mathbb{R}^{N\times N}$. Thus, we can see that differential equations for $\WW^{(\omega+\delta)}$ can be treat as (2$\cdot$N)th-order system of second-order initial-value problems. For second-order initial-value problems we can introduce new unknown functions being derivative of remaining functions. Then we get
(4$\cdot$N)th-order system of first-order initial-value problems for which there exist rich collection of iterative algorithms. Let us focus on the uniqueness and existence in the general case. Namely, recall that every $m$th-order system of first-order initial-value problems can be written in the form of
\begin{equation}
\nonumber
\begin{array}{c}
\frac{dy_1}{dt} = g_1(t,y_1,y_2,...,y_m) \\[9pt]
\frac{dy_2}{dt} = g_2(t,y_1,y_2,...,y_m)\\[9pt]
\vdots\\[9pt]
\frac{dy_m}{dt} = g_m(t,y_1,y_2,...,y_m)
\end{array}
\end{equation}
where for all $i \in \lbrace 1,2,...,m\rbrace$, $g_i$ is assumed to be defined on some set
\[
D_i = \lbrace (t,y_1,...,y_m) : a \leq t \leq b , -\infty < y_k < \infty, \forall k = 1,2,...,m \rbrace.
\]
%\[
%D_i = \lbrace (t,y_1,...,y_m) : a \leq t \leq b \wedge -\infty < y_k < \infty, \forall k = 1,2,...,m \rbrace
%\]
Then the system has a unique solution $y_1(t),y_2(t),...,y_m(t)$, for $a \leq t \leq b$ if all $g_i$'s are continuous on $D_i$ and satisfy a Lipschitz condition with respect to $(y_1,y_2,...,y_m)$.

In the framework of this section, we choose $a = -d$ and $b = t_{max}$ as a upper limit of our approximation. It is also clear that if we choose $\omega$ to be continuous then above sufficient condition holds true. Set the following parameters
\begin{equation}
\nonumber
\begin{split}
&\Delta_{\bo{\sigma}} = \left(\begin{array}{cc} 1.2 & 0 \\ 0 & 2 \end{array}\right), \quad\Delta_{\bo{\mu}} = \left(\begin{array}{cc} 1.75 & 0 \\ 0 & 1.25 \end{array}\right), \quad
\bo{Q} = \left(\begin{array}{cc} -0.4 & 0.4 \\ 0.2 & -0.2 \end{array}\right),\\[10pt]& \gamma_0 = 0.5,\quad \gamma_1 = -0.1, \quad d = 5, \quad t_{max} = 10 \quad \text{ and } \quad \delta = 0.04.
\end{split}
\end{equation}
Figures \ref{Fig31} and \ref{Fig41}  present entries of the numerical approximations of the matrices $\WW^{(\omega + \delta)}$ and $\bo{v}_c^d (x)$ respectively.
\begin{figure}[H]
\centering
\includegraphics[width=12cm]{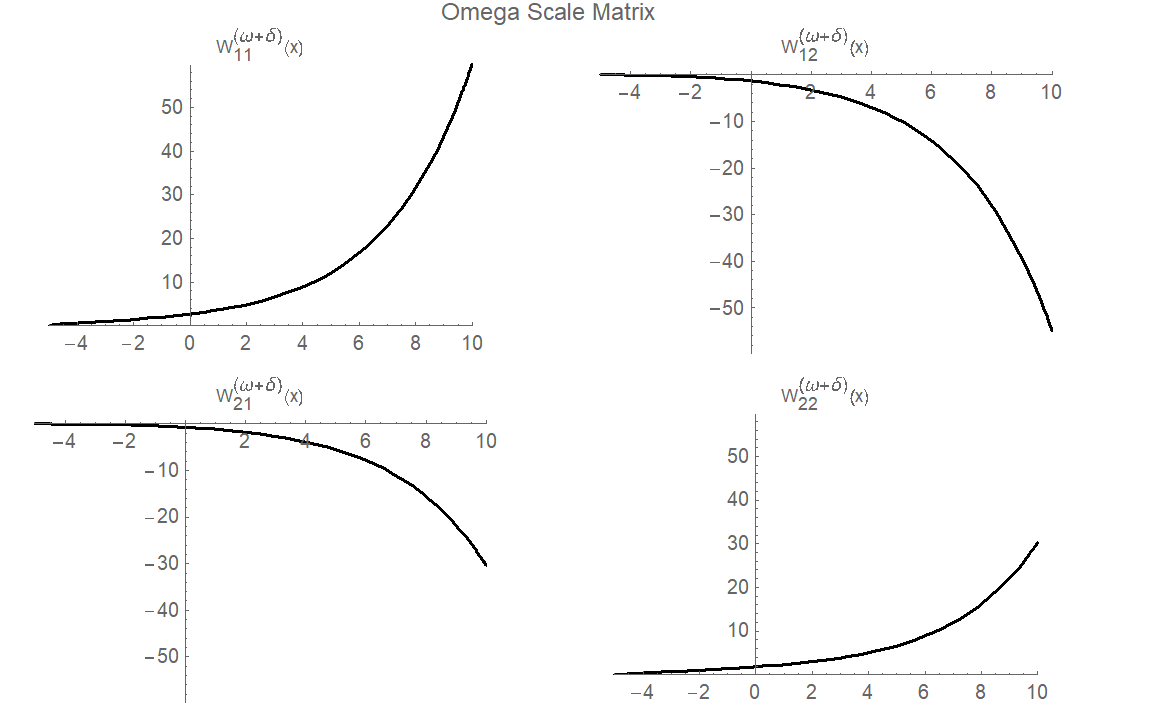}
\caption{Entries of $\omega$-scale matrix function $\WW^{(\omega+\delta)}$}\label{Fig31}
\end{figure}
%It is clear that the main difference between this scale matrix and the others is that it has non zero values in the interval $(-d,0]$,
%due to the definition of omega ruin time. It is mandatory to say that this matrix has also the exponential shape, similar as matrices before.  \\

%Recall that, in spectrally negative L\'evy case, we was looking for strategy such that below supremum was attained
%\[
%\upsilon_{*}(x) = \sup_{\pi \in \Pi} \upsilon_{\pi}(x),
%\]
%where $\upsilon$ is a value function and $\Pi$ is a set of all admissible strategies.
%Thus we need to chose such barrier that maximize $\upsilon$ for all $x.$ Despite that we did not consider optimality of barrier strategy
%in MMBM framework we can look at the above picture and see that in all cells we have shape that allows us to chose such barrier level.
%\baselineskip13pt
\begin{figure}[H]
\centering
\includegraphics[width=12cm]{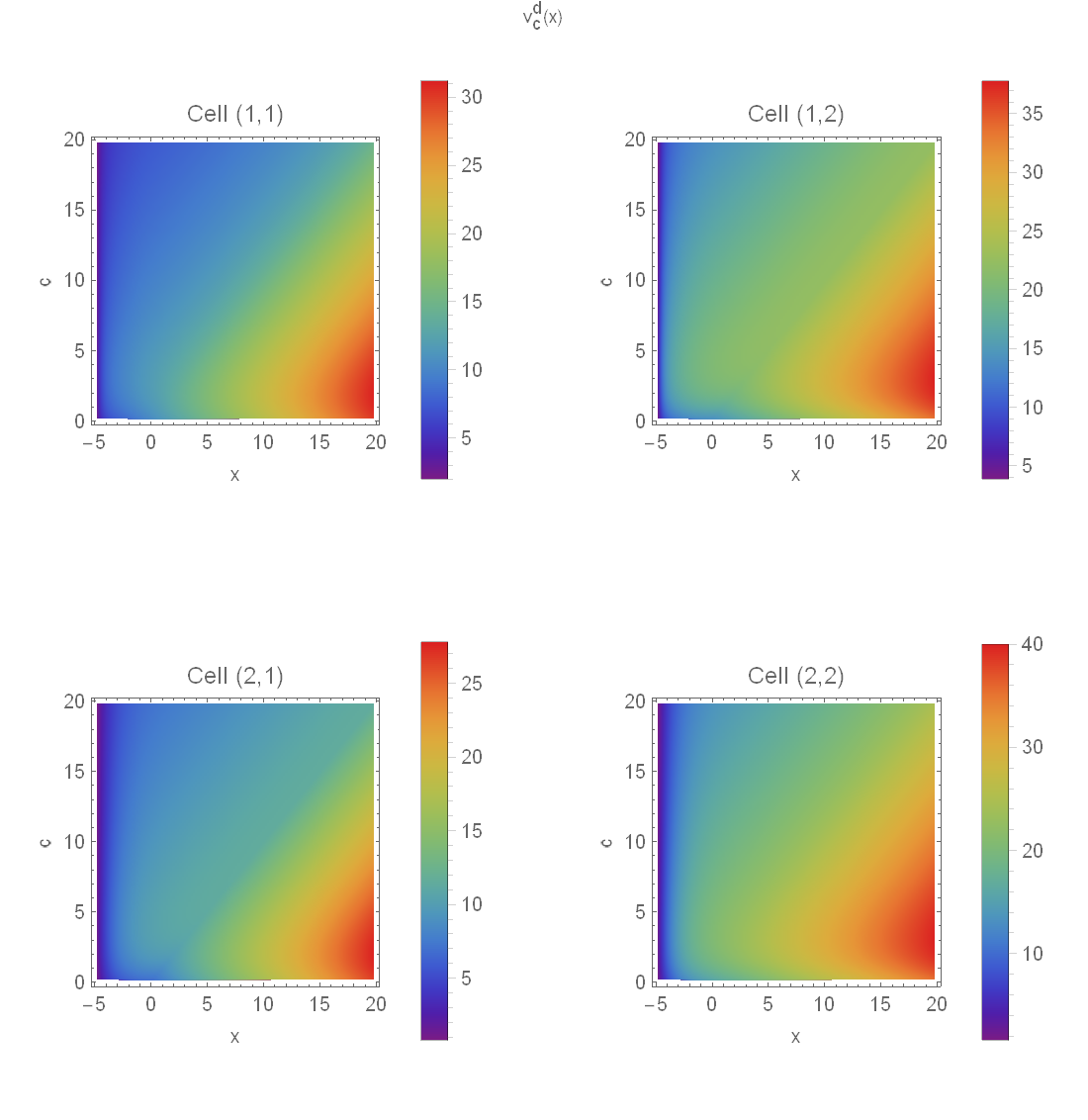}
\caption{Entries of the value matrix function $\bo{v}_c^d(x)$}\label{Fig41}
\end{figure}
\appendix
\section{Proof of Lemma \ref{lem unique}}\label{Lemma2.1}
To prove the uniqueness of the solution, we will show that $\mathbf{H}(x)=\mathbf{0}$ is the only solution to
\begin{equation}
\mathbf{H}(x)=\int_0^{x}\mathbf{W}(x-y)\boldsymbol{\omega}(y)\mathbf{H}(y)dy. \label{HHH}
\end{equation}
%Multiplying both sides of (\ref{HHH}) by $e^{-s_0x}$ and taking the integration with respect to $x$,
Taking the Laplace transform on both sides of (\ref{HHH}) (with an argument $s_0$), we get
%\[
%e^{-s_0x}\mathbf{H}(x)=\int_0^{x}  e^{-s_0(x-y)}\mathbf{W}(x-y)e^{-s_0 y}\boldsymbol{\omega}(y)\mathbf{H}(y)dy.
%\]
\[
\widetilde{\mathbf{H}}(s_0) = \widetilde{\mathbf{W}}(s_0)\  \widetilde{\boldsymbol{\omega} \mathbf{H}}(s_0).
\]
Recall that $\lambda$ is the upper bound of $|\omega_i(y)|$ on $[0,\infty)$ for all $i \in E$. Using (\ref{matrix_W_def}), we
obtain that the matrix norm of $\widetilde{\mathbf{H}}(s_0) $ fulfills the inequality

\begin{equation}
\lVert \widetilde {\mathbf{H}}(s_0) \rVert \le \lambda \lVert \widetilde {\mathbf{W}}(s_0) \rVert  \lVert \widetilde {\mathbf{H}}(s_0) \rVert
=\lambda \lVert {\mathbf{F}}^{-1}(s_0) \rVert  \lVert \widetilde {\mathbf{H}}(s_0) \rVert. \label{inequal}
\end{equation}
%where $\lambda$ is the upper bound of $|\omega_i(y)|$ on $[0,\infty)$ for all $i$, and using the identity
%\[
%\mathbf{F}(\alpha)^{-1}=\int_0^{\infty} e^{-\alpha y}\mathbf{W}(y)dy.
%\]
Next we will show that there exists $s_0$ such that
\begin{equation}
 \lVert {\mathbf{F}}(s) ^{-1}\rVert< \frac{1}{2\lambda}, \qquad \text{ for all $s\ge s_0$}. \label{bound}
\end{equation}
To do so, we recall the expression for $\mathbf{F}(\alpha)$:
\[
\mathbf{F}(\alpha)=\text{diag}(\psi_1(\alpha),\dots,\psi_N(\alpha) )+\mathbf{Q}\circ \mathbb{E}(e^{\alpha U_{ij}}).
\]
Observe that its diagonal goes to infinity, as $\alpha$ goes to infinity, and each element (entry-wise) other than the diagonal is bounded by the (fixed) $q_{ij}$.

We now prove that (using the induction argument with respect to the dimension of $\mathbf{F}(\alpha)$)
\begin{equation*}
{\mathbf{F}}^{-1}(\alpha) \rightarrow \mathbf{0},\text{ as } \alpha \rightarrow \infty. \label{F-1}
\end{equation*}
Define a series sub-matrices of $\mathbf{F}(\alpha)$, for $m=1,2,\ldots, N$,
\[
\mathbf{F}_{m}(\alpha)^{-1}:=\mathbf{F}(\alpha)^{-1}_{m\times m}=\left(\left \{F_{ij}(\alpha)\right\}_{i,j=1}^m\right)^{-1},
\]
and in what follows, we will show that
\begin{equation}
{\mathbf{F}}_m^{-1}(\alpha) \rightarrow \mathbf{0}_{m\times m},\text{ as } \alpha \rightarrow \infty. \label{Finverse}
\end{equation}
Clearly, $\mathbf{F}_{N}(\alpha)^{-1}=\mathbf{F}(\alpha)^{-1}$.

When $m=1$, $\mathbf{F}_{1}(\alpha)^{-1}=\frac{1}{\psi_1(s_0)+q_{11}}$, which makes (\ref{Finverse}) hold obviously, and $s_0$ in (\ref{bound}) is chosen such that $\frac{1}{\psi_1(s_0)+q_{11}} < \frac{1}{2\lambda}$.
Assume (\ref{Finverse}) holds for the dimension $m=k-1$. Then in the dimension $m=k$, we have
\begin{align*}
\mathbf{F}_{k}(\alpha)^{-1}=& \begin{pmatrix}
  \mathbf{A} & \mathbf{B}\\
  \mathbf{C} & \mathbf{D}
 \end{pmatrix}^{-1},
\end{align*}
where
\[
\mathbf{A}_{(k-1)\times (k-1)}=\mathbf{F}_{k-1}(\alpha),
\]
\[
\mathbf{B}_{(k-1)\times 1}=(q_{1 k} \mathbb{E}(e^{\alpha U_{1k}}), \ldots, q_{(k-1) k} \mathbb{E}(e^{\alpha U_{(k-1)k}}))^T,
\]
\[
\mathbf{C}_{1\times(k-1)}= (q_{ k1} \mathbb{E}(e^{\alpha U_{k1}}), \ldots, q_{k(k-1) } \mathbb{E}(e^{\alpha U_{k(k-1)}})),
\]
and
\[
\mathbf{D}_{1\times1}=\psi_k(\alpha)-q_{kk}.
\]

Using the property for the inverse of the block matrix
\[
 \begin{pmatrix}
  \mathbf{A} & \mathbf{B}\\
  \mathbf{C} & \mathbf{D}\\
 \end{pmatrix}^{-1}
 = \begin{pmatrix}
  \mathbf{A}^{-1}+\mathbf{A}^{-1}\mathbf{B}(\mathbf{D}-\mathbf{C}\mathbf{A}^{-1}\mathbf{B})^{-1}\mathbf{C}\mathbf{A}^{-1} & -\mathbf{A}^{-1}\mathbf{B}(\mathbf{D}-\mathbf{C}\mathbf{A}^{-1}\mathbf{B})^{-1}\\
  -(\mathbf{D}-\mathbf{C}\mathbf{A}^{-1}\mathbf{B})^{-1}\mathbf{C}\mathbf{A}^{-1} & (\mathbf{D}-\mathbf{C}\mathbf{A}^{-1}\mathbf{B})^{-1}\\
 \end{pmatrix},
\]
it is easy to see that each block for $\mathbf{F}_{k}(\alpha)^{-1}$ goes to $\mathbf{0}$ as $\alpha\rightarrow \infty$, since
\[
\mathbf{A}^{-1}=\mathbf{F}_{k-1}(\alpha)^{-1}\rightarrow \mathbf{0}_{(k-1)\times(k-1)},
\]
\[
(\mathbf{D}-\mathbf{C}\mathbf{A}^{-1}\mathbf{B})^{-1}=\frac{1}{\psi_k(\alpha)-q_{kk}-\mathbf{C}\mathbf{A}^{-1}\mathbf{B}} \rightarrow 0,
\]
and $\mathbf{B}$, $\mathbf{C}$ have bounded (non-negative) elements.

This completes the proof of \eqref{bound}.
%So far we have proved (\ref{F-1}), which means $\lVert \mathbf{F}(\alpha)^{-1}\rVert \rightarrow 0$ as $\alpha\rightarrow \infty$. We assume that such a finite $s_0$ exists for (\ref{bound}) to %hold, however, the choice of $s_0$ may depend on the matrix norm selected, the matrix $\mathbf{Q}$, and the Laplace exponents $\{\psi_i(\cdot)\}_{i=1}^{k}$.

Plugging (\ref{bound}) into (\ref{inequal}) gives:
\[
\lVert \widetilde {\mathbf{H}}(s_0) \rVert =0, \text{ i.e., } \mathbf{H}(x)=\mathbf{0},
\]
which completes the proof of uniqueness of the solution of Equation \eqref{Hh}.

To prove the existence of solution of of Equation \eqref{Hh}, we construct a series of matrices $\{\mathbf{H}_m\}$ which converge to the unique solution. Define the operator $\mathcal{G}$ on a matrix: for $z>0$,
\[
\mathcal{G}\widetilde{\mathbf{K}}(z)=\int_0^{\infty} e^{-zx}\int_0^{x} e^{-s_0(x-y)}\mathbf{W}(x-y) \boldsymbol{\omega}(y)\mathbf{K}(y)dydx= \widetilde {\mathbf{W}}(s_0+z)\  \widetilde {\boldsymbol{\omega} \mathbf{K}}(z) ,
\]
\[
\mathcal{G}^{(m+1)} \widetilde{\mathbf{K}}(z) =\mathcal{G} (\mathcal{G}^{(m)})\widetilde{\mathbf{K}}(z),
\]
\[
\widetilde{\mathbf{H}}_0(z)=\int_{0}^{\infty} e^{-zx}e^{-s_0x}\mathbf{h}(x)dx=\widetilde{\mathbf{h}}_0(s_0+z), \qquad \widetilde{\mathbf{H}}_{m+1}(z)=\widetilde{\mathbf{H}}_0(z)+\mathcal{G}\widetilde{\mathbf{H}}_m(z).
\]

Then we have $\mathcal{G}$ is a linear operator such that $\lVert\mathcal{G}\widetilde{\mathbf{K}}(z) \rVert < \frac{1}{2}\lVert \widetilde{\mathbf{K}}(z) \rVert $ for $z>0$.
Therefore, for $m>l$, we have
\[
\lVert \widetilde{\mathbf{H}}_m(z) - \widetilde{\mathbf{H}}_l(z) \rVert = \lVert \sum_{k=l+1}^{m}\mathcal{G}^{(k)}\widetilde{\mathbf{H}}_0(z) \rVert < 2^{-l}  \lVert \widetilde{\mathbf{H}}_0(z) \rVert,
\]
which means $\{\widetilde{\mathbf{H}}_{m}(z),z>0\}_{m\ge0}$ forms a Cauchy sequence (entry-wise) that admits a limit $\widetilde{\mathfrak{H}}(z)$ for any $z>0$ satisfying:
\[
\widetilde{\mathfrak{H}}(z)=\widetilde{\mathbf{H}}_0(z)+\mathcal{G}\widetilde{\mathfrak{H}}(z)=\widetilde{\mathbf{h}}_0(s_0+z)+ \widetilde {\mathbf{W}}(s_0+z)\  \widetilde {\boldsymbol{\omega} \mathfrak{H}}(z).
\]

Using the uniqueness of Laplace transform, we have
\[
{\mathfrak{H}}(x)=e^{-s_0x}{\mathbf{h}}_0(x)+  \int_0^{x} e^{-s_0(x-y)}\mathbf{W}(x-y)\boldsymbol{ \omega}(y)\mathfrak{H}(y)dy,
\]
which shows that $\mathbf{H}(x)=e^{s_0x}{\mathfrak{H}}(x)$ is the solution to (\ref{Hh}).

As for the second statement in this lemma, we see that if $\mathbf{H}$ satisfies (\ref{Hhdelta}), by letting $\delta=0$, we obtain (\ref{Hh}) immediately. Now we only need to show that if $\mathbf{H}$  is the solution to (\ref{Hh}), it is also the solution to (\ref{Hhdelta}).
We convolute both sides of (\ref{Hh}) with $\delta \mathbf{W}^{(\delta)}$ (on the left),
\begin{align*}
\delta \mathbf{W}^{(\delta)}* \mathbf{H}(x)&=\delta \mathbf{W}^{(\delta)}*\mathbf{h}(x)+\delta \mathbf{W}^{(\delta)}* \mathbf{W} *(\boldsymbol{\omega}\mathbf{H})(x)\\
&=\delta \mathbf{W}^{(\delta)}*\mathbf{h}(x)+( \mathbf{W}^{(\delta)}- \mathbf{W})*(\boldsymbol{\omega}\mathbf{H})(x)
\end{align*}
where the last step using the identity $ \mathbf{W}^{(\delta)}- \mathbf{W}= \delta \mathbf{W}^{(\delta)}* \mathbf{W}$ (which can be easily seen from the Laplace transform). Therefore,
\[
\mathbf{H}(x)=\mathbf{h}(x)+\delta \mathbf{W}^{(\delta)}*\mathbf{h}(x)+ \mathbf{W}^{(\delta)}* ((\boldsymbol{\omega}-\delta\mathbf{I})\mathbf{H})(x),
\]
which completes the proof.
\halmos

\section{Proofs of main results}\label{Proof Main}

\subsection{Proof of Theorem \ref{Two-sided}}
\subsubsection{Proof of the case $(i)$}\label{AppendixB11}

In what follows, we prove the case of $d=0$, and then the general result holds true using the shifting argument as well as the identity (\ref{shiftw}).

First, applying the strong Markov property of $X$ at $\tau_y^+$ and using the fact that $X$ has no positive jumps, we get that:
\begin{equation}
\mathbf{A}^{(\omega)}(x,z)=\mathbf{A}^{(\omega)}(x,y)\mathbf{A}^{(\omega)}(y,z), \label{AA}
\end{equation}
for all $0 \le x\le y\le z.$

Following the similar argument as in Li and Palmowski \cite{LP16}, we recall that $\lambda>0$
is the arbitrary upper bound of $\omega_i(x)$ (for all $x \in \mathbb{R}$ and $1\le i \le N$). Let $\Psi=\{\Psi_{t},t\ge 0\}$ be a Poisson point process with a characteristic measure
$\mu(dt,dy)=\lambda dt \; \frac{1}{\lambda}1_{\{[0,\lambda]\}}(y)dy$. Hence $\Psi=\{(T_k,M_k),k=1,2,\dots\}$ is a doubly stochastic marked Poisson process with jump intensity
$\lambda$, jumps epochs $T_k$ and marks $M_k$ being uniformly distributed on $[0,\lambda]$.
Moreover, we construct $\Psi$ to be independent of $X$. Therefore, for $T^{\omega}:=\inf{\{T_k>0: M_k < \omega_{J_{T_k}}(X_{T_k}); \mbox{ for } k \geq 1 \}}$, we have
\begin{align*}
A_{ij}^{(\omega)}(x,c)
=& \mathbb{P}_{x,i}\left( \tau_c^+< \tau_0^- \wedge T^{\omega}, J_{\tau_c^+}=j  \right)\\
=& \mathbb{P} _{x,i}\left(\#_{k} \{ M_{k}<\omega_{J_{T_k}} (X_{T_{k}} ) \mbox{ for } T_{k}< \tau_c^+ ,\tau_c^+<\tau_0^- , J_{\tau_c^+}=j   \}=0 \right).
\end{align*}
In this case, there are two scenarios following:
either there is no $T_k$ which occurs before reaching level $c$ or the first jump time $T_1$ occurs in state $m$ and the process renews from state $m$.
Hence:
\begin{align*}
A_{ij}^{(\omega)}(x,c)=& \mathbb{P} _{x,i}(T_1>\tau_c^+, \tau_c^+<\tau_0^-, J_{\tau_c^+}=j)\\
&+\sum_{m=1}^{N}\mathbb{E}_{x,i} \left[ A_{mj}^{(\omega)}(X_{T_1},c), T_1<\tau_c^+\wedge \tau_0^-, M_1>\omega_m(X_{T_1}), J_{T_1}=m \right]\\
=&\mathbb{E} _{x,i}[e^{-\lambda \tau_c^+}; \tau_c^+<\tau_0^-, J_{\tau_c^+}=j]\\
&+\int_0^{\infty} \sum_{m=1}^{N} \mathbb{E}_{x,i} \left[ X_{T_1}\in dy, T_1<\tau_c^+\wedge \tau_0^-, J_{T_1}=m \right]   \frac{\lambda-\omega_m(y)}{\lambda}
A_{mj}^{(\omega)}(y,c),
\end{align*}
which is equivalent to
\begin{align*}
\mathbf{A}^{(\omega)}(x,c)&=\mathbb{E} _{x}[e^{-\lambda \tau_c^+}, \tau_c^+<\tau_0^-, J_{\tau_c^+}]\\&+\int_0^{c}
\mathbb{E}_{x} \left[ X_{T_1}\in dy, T_1<\tau_c^+\wedge \tau_0^-, J_{T_1} \right]\frac{1}{\lambda}  \left( \lambda\mathbf{I}-\boldsymbol{\omega}(y) \right)
\mathbf{A}^{(\omega)}(y,c),
\end{align*}
where
\[
\mathbb{E} _{x}[e^{-\lambda \tau_c^+}, \tau_c^+<\tau_0^-, J_{\tau_c^+}]=\mathbf{W}^{(\lambda)}(x)\mathbf{W}^{(\lambda)}(c)^{-1}
\]
and
\[
\frac{1}{\lambda}\mathbb{E}_{x} \left[ X_{T_1}\in dy, T_1<\tau_c^+\wedge \tau_0^-, J_{T_1} \right] =
\left(\mathbf{W}^{(\lambda)}(x)\mathbf{W}^{(\lambda)}(c)^{-1}\mathbf{W}^{(\lambda)}(c-y)-\mathbf{W}^{(\lambda)}(x-y)\right) dy,
\]
are given in Ivanovs and Palmowski \cite{IP12} and Ivanovs \cite{I14}, respectively.

Taking the last increment to the other side of the above equality and applying relation (\ref{AA}) gives
\begin{align}\label{relation}
&\left(\mathbf{I}+\int_0^{x} \mathbf{W}^{(\lambda)}(x-y)  \left( \lambda\mathbf{I}-\boldsymbol{\omega}(y) \right)\mathbf{A}^{(\omega)}(y,x)dy\right)\mathbf{A}^{(\omega)}(x,c)\\ \nonumber
=&\mathbf{W}^{(\lambda)}(x)\mathbf{W}^{(\lambda)}(c)^{-1}\left( \mathbf{I}+\int_0^{c}\mathbf{W}^{(\lambda)}(c-y)
\left( \lambda\mathbf{I}-\boldsymbol{\omega}(y) \right)\mathbf{A}^{(\omega)}(y,c)dy\right).
\end{align}

Define
\begin{equation}
\mathcal{W}^{(\omega)}(x)^{-1}:=\mathbf{W}^{(\lambda)}(x)^{-1}\left(\mathbf{I}+\int_0^{x}  \mathbf{W}^{(\lambda)}(x-y) \left( \lambda\mathbf{I}-\boldsymbol{\omega}(y) \right)
\mathbf{A}^{(\omega)}(y,x)dy\right) \label{AW}
\end{equation}

%Define
%\begin{equation}
%\mathcal{W}^{(\omega)}(x)=\left(\mathbf{I}+\int_0^{x}  \mathbf{W}^{(\lambda)}(x-y) \left( \lambda\mathbf{I}-\boldsymbol{\omega}(y) \right)
%\mathbf{A}^{(\omega)}(y,x)dy\right)^{-1}\mathbf{W}^{(\lambda)}(x) \label{AW}
%\end{equation}

and then we obtain the required identity
\[
\mathbf{A}^{(\omega)}(x,c)=\mathcal{W}^{(\omega)}(x)\mathcal{W}^{(\omega)}(c)^{-1}.
\]

The proof of the invertibility of matrix $\mathcal{W}^{(\omega)}(x)^{-1}$ is deferred to Proposition \ref{invertibility}.

After replacing $\mathbf{A}^{(\omega)}(y,x)=\mathcal{W}^{(\omega)}(y)\mathcal{W}^{(\omega)}(x)^{-1}$ in (\ref{AW}), we have
\begin{align*}
\mathbf{W}^{(\lambda)}(x)=&\left(\mathbf{I}+\int_0^{x}\mathbf{W}^{(\lambda)}(x-y) \left( \lambda\mathbf{I}-\boldsymbol{\omega}(y) \right)\mathbf{A}^{(\omega)}(y,x)dy\right)
\mathcal{W}^{(\omega)}(x)\\
=&\mathcal{W}^{(\omega)}(x)+\int_0^{x} \mathbf{W}^{(\lambda)}(x-y) \left( \lambda\mathbf{I}-\boldsymbol{\omega}(y) \right)\mathcal{W}^{(\omega)}(y)dy
\end{align*}

Now using the identity
\[
\mathbf{W}^{(\delta)}-\mathbf{W}=\delta\mathbf{W}* \mathbf{W} ^{(\delta)},
\]
it is easy to show
\begin{align*}
\mathcal{W}^{(\omega)}(x)=\mathbf{W}(x)+\int_0^{x}\mathbf{W}(x-y)  \boldsymbol{\omega}(y)\mathcal{W}^{(\omega)}(y)\;dy.
\end{align*}
\halmos
%\begin{flushright}
%$ \blacksquare$
%\end{flushright}

\begin{proposition} \label{invertibility}
The matrix $\mathcal{W}^{(\omega)}(x)^{-1}$ is invertible for any $x > 0$.
\end{proposition}

\begin{proof}
From (\ref{AW}), one can see that it is enough to prove that the matrix
$$\mathbf{P}(x):= \mathbf{I}+\int_0^{x}  \mathbf{W}^{(\lambda)}(x-y) \left( \lambda\mathbf{I}-\boldsymbol{\omega}(y)\right)\mathbf{A}^{(\omega)}(y,x)dy$$  is invertible
for every $x\geq 0$. %It is obvious that $\mathbf{P}(0)=\mathbf{I}$ is invertible.
Using similar argument as in \cite{KP08}, note that for all $y > 0$ there exists some $N \times N$ sub-stochastic
invertible intensity matrix $\bo{\Lambda}^{\omega, *}(y)$ such that
\begin{equation}\label{Sub-stoch}
\mathbb{P}_x\left( \tau_c^+< \tau_0^- \wedge T^{\omega}, J_{\tau_c^+}| J_0  \right)= \exp\left(\int_x^c \bo{\Lambda}^{\omega, *}(y)\; dy\right).
\end{equation}
This observation implies that the matrix $\mathbf{A}^{(\omega)}(x,c)$ is invertible for any $x,c \in \mathbb{R}_{+}$ such that $0<x\leq c$.
The matrix $\mathbf{A}^{(\omega)}(x,c)$ is also continuous (entry wise) with respect of $c$.
Now, assume that there exists $c>0$ such that matrix $\mathbf{P}(x)$ is invertible for some  $0<x<c$
and is singular for $x=c$. Then from relation (\ref{relation}) we get contradiction, because the left-hand side of it is
invertible (as a product of invertible matrices) and the right-hand side is singular from the assumption.
Hence, only two scenarios are possible: the matrix $\mathbf{P}(x)$ is invertible for all $x > 0$ or it is singular for all $x> 0$.
Finally, since $\mathbf{P}(0)=\mathbf{I}$ and $\mathbf{P}(x)$ is continuous in $x \geq 0$ we obtain that $\mathbf{P}(x)$ must be invertible for
all $x\geq 0$.
\end{proof}\\

\subsubsection{Proof of the case $(ii)$}

Let $\{(X_t,J_t)\}_{t\ge 0 }$ be a MAP with the lifetime $\xi$, transition probabilities and $q$-resolvent measures, given,
respectively by
\[
Q_{t,ij} f_j(x)=\mathbb{E}_{x,i}\left[ f_j(X_t), t < \xi, J_t=j \right]
\]
and
\[
K^{(q)}_{ij} f_j(x)= \int_0^{\infty} e^{-qt} Q_{t,ij}f_j(x) dt,
\]
where $\{f_j\}_{j=1}^N$ is a set of nonnegative, bounded, continuous functions on $\mathbb{R}$ such that \\
$\sup_{i,j} K^{(0)}_{ij}f_j(x)<\infty$. %Let $\{\omega_i\}_{i=1}^N$ be a set of nonnegative and bounded functions on $\mathbb{R}$ and
 Then the $\omega$-type resolvent ${K}^{(\omega)}_{ij}$ is defined by
\[
K^{(\omega)}_{ij} f_j(x):=\int_0^{\infty} Q_{t,ij}^{(\omega)} f_j(x)dt,
\]
where
\[
Q_{t,ij}^{(\omega)} f_j(x):=\mathbb{E}_{x,i} \left[ \exp\left ( -\int_0^{t} \omega_{J_s} (X_s) ds\right) f_{j}(X_t);t<\xi, J_{t}=j \right].
\]

The next lemma is a helpful tool used further to get the representation of the matrix $\bo{B}^{(\omega)}(x,c)$.

\begin{lemma} \label{lemk}
The matrix $\mathbf{K}^{(\omega)}\bo{f}(x)=\{K^{(\omega)}_{ij} f_j(x)\}_{i,j=1}^N$ satisfies the following equality:
\[
\mathbf{K}^{(\omega)} \bo{f}(x)= \mathbf{K}^{(0)}\left ( \boldsymbol{f} - \boldsymbol{\omega}\mathbf{K}^{(\omega)}\bo{f}\right)(x),
\]
%where the matrix $\mathbf{P}(x)$ is defined by
%\[
%P_{ij}(x)=\mathbb{E}_{x,i} \left[ \xi<\infty, J_{\xi}=j \right].
%\]
where $\boldsymbol{f}=\textrm{diag}(f_1,...,f_N)$.
\end{lemma}

\begin{proof}
As before without loss of generality, we assume that $\omega_i(x)$ is bounded by some $\lambda>0$ for all $x \in \mathbb{R}$ and $i \in E$.
The finiteness of $K^{(\omega)}_{ij} f_j(x)$ comes from the fact that $K^{(\omega)}_{ij} f_j(x)<K^{(0)}_{ij} f_j(x)$ for all $1\le i \le N$.
Using similar arguments as in the proof in \ref{AppendixB11}, we have
\begin{align*}
Q_{t,ij}^{(\omega)} f_j(x)=& \mathbb{E}_{x,i}\left [ f_{J_t}(X_t); t< \xi \mbox{ and } M_k> \omega_{J_{T_k}}(X_{T_k}) \mbox{ for all }T_k<t, J_{t}=j\right]\\
=& \mathbb{E}_{x,i}\left [ f_{J_t}(X_t); t< \xi, T_1>t, J_{t}=j\right]\\
&+\sum_{l=1}^{N} \int_0^{t} \mathbb{E}_{x,i}\left [ Q_{t-s,lj}^{(\omega)} f_j(X_s), M_1> \omega_{l}(X_{s}), J_s=l\right] \mathbb{P}(T_1 \in ds)  \\
=& \mathbb{E}_{x,i}\left [ e^{-\lambda t}f_{j}(X_t); t< \xi, J_t=j\right]\\
&+\sum_{l=1}^{N} \int_0^{t} \mathbb{E}_{x,i}\left [ (\lambda -\omega_l (X_s))Q_{t-s,lj}^{(\omega)} f_j(X_s), J_s=l\right] e^{-\lambda s}ds \\
=& Q_{t,ij}^{(\lambda)} f_{j}(x)+ \sum_{l=1}^{N}\int_0^{t} Q_{s,il}^{(\lambda)} \left ((\lambda-\omega_l)Q_{t-s,lj}^{(\omega)} f_j\right)(x)ds.
\end{align*}
Note that the superscript $\lambda$ denotes a counterpart for fixed $\omega_i(x)\equiv\lambda$.
%where note that superscript $\lambda$ is a bit different from a counterpart for fixed $\omega(x)=\lambda$ in the matrix form (which are the same in one-dimensional case).
Equivalently, in a matrix form, we have
\[
\mathbf{Q}_{t}^{(\omega)} \bo{f}(x)= \mathbf{Q}_t^{(\lambda)}\boldsymbol{f}(x)+ \int_0^{t} \mathbf{Q}_s^{(\lambda)}\left((\lambda\mathbf{I}-\boldsymbol{\omega})\mathbf{Q}_{t-s}^{(\omega)}\bo{f}\right)(x) ds,
\]
where by matrix compounding, we mean $\left (\mathbf{A}(\mathbf{B})(x)\right)_{ij}=\sum_{m=1}^N A_{im} B_{mj}(x)$.
%$((\mathbf{A} \circ \mathbf{B})(x))_{ij}= \left (\mathbf{A}(\mathbf{B})(x)\right)_{ij}=\sum_{m=1}^N A_{im} B_{mj}(x)$.

Thus,
\begin{equation}
\mathbf{K}^{(\omega)}\bo{f}(x)=\int_0^{\infty} \mathbf{Q}_{t}^{(\omega)} \bo{f}(x)dt=\mathbf{K}^{(\lambda)}\boldsymbol{f}(x)+ \mathbf{K}^{(\lambda)}\left((\lambda\mathbf{I}-\boldsymbol{\omega})
\mathbf{K}^{(\omega)}\bo{f}\right)(x). \label{Komega}
\end{equation}
Using the resolvent identity $\lambda \mathbf{K}^{(0)}(\mathbf{K}^{(\lambda)})= \mathbf{K}^{(0)}- \mathbf{K}^{(\lambda)}$, we have
\begin{align}
\lambda \mathbf{K}^{(0)} \left(\mathbf{K}^{(\omega)}\bo{f} \right) (x)=& \lambda \mathbf{K}^{(0)}( \mathbf{K}^{(\lambda)}\boldsymbol{f})(x)+
\lambda \mathbf{K}^{(0)}\left( \mathbf{K}^{(\lambda)}
\left((\lambda\mathbf{I}-\boldsymbol{\omega})\mathbf{K}^{(\omega)}\bo{f}\right)\right)(x) \nonumber\\
=&(\mathbf{K}^{(0)}- \mathbf{K}^{(\lambda)})\boldsymbol{f}(x)+ (\mathbf{K}^{(0)}- \mathbf{K}^{(\lambda)}) \left((\lambda\mathbf{I}-\boldsymbol{\omega})\mathbf{K}^{(\omega)}\bo{f}\right)(x).\label{Kome}
\end{align}

%\begin{align}
%\lambda \mathbf{K}^{(0)} \left(\mathbf{K}^{(\omega)}\bo{f} \right) (x)=& \lambda \mathbf{K}^{(0)}\circ \mathbf{K}^{(\lambda)}\boldsymbol{f}(x)+\lambda \mathbf{K}^{(0)}\circ \mathbf{K}^{(\lambda)}
%\left((\lambda\mathbf{I}-\boldsymbol{\omega})\mathbf{K}^{(\omega)}\bo{f}\right)(x) \nonumber\\
%=&(\mathbf{K}^{(0)}- \mathbf{K}^{(\lambda)})\boldsymbol{f}(x)+ (\mathbf{K}^{(0)}- \mathbf{K}^{(\lambda)}) \left((\lambda\mathbf{I}-\boldsymbol{\omega})\mathbf{K}^{(\omega)}\bo{f}\right)(x).\label{Kome}
%\end{align}
Comparing (\ref{Komega}) with (\ref{Kome}) completes the proof.
\end{proof}

%{\\ \color{blue} Czy nie powinnismy przeniesc definicje macierzy $\bo{B}$(x) $($tj. granicy $\lim_{c \rightarrow \infty} B_{ij}(x,c)$ $)$ gdzies poza ten dowod? To akurat jest macierz, ktora dobrze by bylo "pokazac"\\}
%Let us proceed to the second main theorem of this paper
Now we can proceed the proof of the case $(ii)$. Again we prove the case of $d=0$, and then the general result holds true using the shifting argument as well as the identity (\ref{shiftw}).
 %\label{thm2} For $x\le c$,
%\[
%\mathbf{B}(x,c)=\mathcal{Z}^{(\omega)}(x)-\mathcal{W}^{(\omega)}(x) \mathcal{W}^{(\omega)}(c)^{-1}\mathcal{Z}^{(\omega)}(c).\]

For $i,j \in E$, define
\begin{equation}
B_{ij}^{(\omega)}(x):=\lim_{c\rightarrow\infty}B_{ij}^{(\omega)}(x,c)=
\mathbb{E}_{x,i} \left[e^{-\int_0^{\tau_0^-}\omega_{J_s}(X_s)ds},\tau_0^- < \infty, J_{\tau_0^-}=j  \right]. \label{Bx}
\end{equation}
Note that for any $i,j \in E$ and $x,c \in \mathbb{R}$ such that $x<c$ matrix function $B_{ij}^{(\omega)}(x,c)$ is monotone in $c$, and it is bounded by
$0 \leq B_{ij}^{(\omega)}(x,c)\leq \mathbb{P}_{x,i} \left(\tau_0^- < \tau_c^+, J_{\tau_0^-}=j\right)\leq 1$,
so the limit in (\ref{Bx}) exists and is finite.
%so we have
%\[
%\mathbf{B}(x)=\mathbf{B}(x,c)+\mathbf{A}(x,c)\mathbf{B}(c),
%\]
%or equivalently,
The strong Markov property and spectrally negativity of $X$ give that
\begin{equation}
\mathbf{B}^{(\omega)}(x,c)=\mathbf{B}^{(\omega)}(x)-\mathbf{A}^{(\omega)}(x,c)\mathbf{B}^{(\omega)}(c).\label{Bxc}
\end{equation}

To identify $\mathbf{B}^{(\omega)}(x)$, we use lemma \ref{lemk} with $\xi=\tau_0^-$ and $\boldsymbol{f}(\cdot)=\boldsymbol{\omega}(\cdot)$. Hence

\begin{align}\label{Invert_B}\nonumber
\mathbf{I}(x)-\mathbf{B}^{(\omega)}(x)=& \mathbb{E}_{x} \left[ \int_0^{\tau_0^-} \omega_{J_t}(X_t) \exp\left( -\int_0^t \omega_{J_s}(X_s) ds\right)dt, t<\tau_0^-, J_{t}\right]\\\nonumber
=&\int_0^{\infty}\mathbb{E}_{x} \left[  \omega_{J_t}(X_t) \exp\left( -\int_0^t \omega_{J_s}(X_s) ds\right), t<\tau_0^- ,J_{t}\right]dt\\
=&\int_0^{\infty} \left( \mathbf{W}(x)e^{\mathbf{R}y}-\mathbf{W}(x-y) \right) \left [ \boldsymbol{\omega}(y)-\boldsymbol{\omega}(\mathbf{I}-\mathbf{B}^{(\omega)})(y)\right]dy,
\end{align}

where the potential measure
$$\mathbf{K}^{(0)} \left(\ind_{(0,\infty)}(X_t \in dy)\right)(x)=\boldsymbol{U}_{(0,\infty)}(x,dy)=\left( \mathbf{W}(x)e^{\mathbf{R} y}-
\mathbf{W}(x-y) \right)dy,$$ was obtained in \cite{I14} with $\mathbf{R}=\mathbf{R}^0$. We may rewrite it as
\begin{equation}\label{nowylabel}
\mathbf{B}^{(\omega)}(x)=\mathbf{I}(x)- \mathbf{W}(x) \mathbf{C}_{B^{(\omega)}} +\int_0^{x} \mathbf{W}(x-y)\boldsymbol{ \omega}(y)\mathbf{B}^{(\omega)}(y)dy,
\end{equation}
where \begin{equation}\label{CBomega} \mathbf{C}_{B^{(\omega)}}=\int_0^{\infty} e^{\mathbf{R}y}\boldsymbol{\omega}(y)\mathbf{B}^{(\omega)}(y)dy.\end{equation}
Note that $0\leq B^{(\omega)}_{ij}(y)\leq 1$ and recall that $0\leq \omega_i(x)\leq\lambda$.
Hence last increment on the right hand side of equation \eqref{nowylabel} is finite and then matrix $\mathbf{C}_{B^{(\omega)}}$ is well defined and finite.

From the definitions of $\omega$-scale matrices
%\[
%\mathbf{P}(x)=\mathbf{I}- \mathbb{P}_x(\tau_0^-=\infty)=\mathbf{I}-\mathbf{W}(x) \mathbf{W}(\infty)^{-1}:=\mathbf{I}-\mathbf{W}(x) \mathbf{C}_{\infty},
%\]
we have
\begin{equation}\label{B_omega}
\mathbf{B}^{(\omega)}(x)=\mathcal{Z}^{(\omega)}(x)-\mathcal{W}^{(\omega)}(x)\mathbf{C}_{B^{(\omega)}}.
\end{equation}
Equation (\ref{Bxc}) completes the proof.
\halmos

\subsection{Proof of Corollary \ref{One-sided}}
\subsubsection{Proof of the case $(i)$}
First we will prove that
\begin{equation}\label{one-sided_(i)}
\lim_{d \to -\infty} \mathcal{W}^{(\omega)}(x,d) \mathcal{W}^{(\omega)}(c,d)^{-1}= \mathcal{H}^{(\omega)}(x)\mathcal{H}^{(\omega)}(c)^{-1}.
\end{equation}

Then the result will follow from Theorem \ref{Two-sided}(i). Recall that for $x\geq d$ and any fixed $\beta \geq 0$ we have:
$$\mathcal{W}^{(\omega)}(x,d)=\mathbf{W}^{(\beta)}(x-d)+\int_{0}^{x} \mathbf{W}^{(\beta)}(x-z) (\boldsymbol{\omega}(z)- \beta \mathbf{I})\mathcal{W}^{(\omega)}(z,d)dz. $$
Moreover, for $x=0$, %\textcolor{red}{for $u \in [0, d)$,}
$$ \mathcal{W}^{(\omega)}(0,d)e^{-\mathbf{R}^{\beta}d}=  \mathbf{W}^{(\beta)}(-d) e^{-\mathbf{R}^{\beta}d}.$$
Hence from \eqref{lim_W} we have
$$\lim_{d \to -\infty}\mathcal{W}^{(\omega)}(0,d)e^{-\mathbf{R}^{\beta}d}=\lim_{d \to -\infty}\mathbf{W}^{(\beta)}(-d)e^{-\mathbf{R}^{\beta}d}=  \mathbf{L^{\beta}} . $$
%\textcolor{red}{$$\lim_{d \to -\infty}\mathbf{W}^{(\beta)}(u-d)e^{-\mathbf{R}^{\beta}d}=  \mathbf{L^{\beta}} e^{-\mathbf{R}^{\beta}u}. $$}
From Theorem \ref{Two-sided}(i), for $x>0$,
$$\mathbb{E}\left[e^{-\int_0^{\tau_x^+}\omega_{J_s}(X_s)ds}, \tau_x^+<\tau_d^-, J_{\tau_x^+}|J_0 \right]\mathcal{W}^{(\omega)}(x,d)= \mathcal{W}^{(\omega)}(0,d).$$
Since the above expectation is increasing with respect to $d$ the following limit is well-defined and finite for every $x>d$:
\begin{align*}
\lim_{d \to -\infty} &\mathbb{E}\left[e^{-\int_0^{\tau_x^+}\omega_{J_s}(X_s)ds}, \tau_x^+<\tau_d^-, J_{\tau_x^+}|J_0 \right]
 \mathcal{W}^{(\omega)}(x,d)e^{-\mathbf{R}^{\beta}d}\\=&
\mathbb{E}\left[e^{-\int_0^{\tau_x^+}\omega_{J_s}(X_s)ds}, \tau_x^+<\infty, J_{\tau_x^+}|J_0 \right]\lim_{d \to -\infty}
\mathcal{W}^{(\omega)}(x,d)e^{-\mathbf{R}^{\beta}d}=\mathbf{L^{\beta}}.
\end{align*}
Note also that, since matrix $\mathbf{L^{\beta}}$ is invertible as it was note above Equation \eqref{lim_W},
from above equation it follows that the matrix
$\lim_{d \to -\infty}
\mathcal{W}^{(\omega)}(x,d)e^{-\mathbf{R}^{\beta}d}$ is also invertible.
Taking
$$\mathcal{H}^{(\omega)}(x):=\lim_{d \to -\infty} \mathcal{W}^{(\omega)}(x,d) e^{-\mathbf{R}^{\beta}d} (\mathbf{L^{\beta}})^{-1}.$$
completes the proof of the first part of the corollary.
To show that the above form of $\mathcal{H}^{(\omega)}(x)$ satisfies (\ref{matrixH}), note that
$$ \mathcal{W}^{(\omega)}(x,d) e^{-\mathbf{R}^{\beta}d}=\left(\mathbf{W}^{(\beta)}(x-d)+\int_{0}^{x}
\mathbf{W}^{(\beta)}(x-z) (\boldsymbol{\omega}(z)- \beta \mathbf{I})\mathcal{W}^{(\omega)}(z,d)dz \right) e^{-\mathbf{R}^{\beta}d}. $$
Then by taking the limit $d \to -\infty$ and applying the dominated convergence theorem the result follows.
\halmos

\subsubsection{Proof of the case $(ii)$}
The proof follows by taking the limit (\ref{Bx}), which exists and is finite. %Finally, the identity (\ref{B_omega})
%with $\mathbf{C}_{B^{(\omega)}}=\mathbf{C}_{\mathcal{W}(\infty)^{-1}\mathcal{Z}(\infty)}$ completes the proof. %$\blacksquare$
%Theorem \ref{Two-sided} $(i)$ one can observe that matrix function $\mathcal{W}^{(\omega)}(c)^{-1}\mathcal{Z}^{(\omega)}(c)$
%is monotonic in $c$. Hence the limit $\lim_{c \to \infty} \mathcal{W}^{(\omega)}(c)^{-1}\mathcal{Z}^{(\omega)}(c)$ exists.
%Moreover, it is finite, because for any $i,j \in E$ and $x,c \in \mathbb{R}$ such that $x<c$, we have that
%$0 \leq \mathbf{B}^{(\omega)}_{ij}(x,c)\leq \mathbb{P}_{x,i} \left(\tau_0^- < \tau_c^+, J_{\tau_0^-}=j\right)\leq 1$.  \blacksquare$
Moreover, 
%Next, one can observe that the matrix $\mathbf{I}(x)-\mathbf{B}^{(\omega)}(x)$ is invertible, because it is strictly diagonally dominant,
%see \cite{BR60} for details. Then from (\ref{Invert_B}) the matrix $\boldsymbol{\omega}(y)\mathbf{B}^{(\omega)}(y)$
%is invertible for any $y \geq 0$ and hence the matrix $\mathbf{C}_{B^{(\omega)}}$ is also invertible as a
%integral from the product of invertible matrices. Finally, recalling that from \eqref{CBomega}, the matrix $\mathbf{C}_{B^{(\omega)}}<\infty$  is well-defined,
%then
the limit
$$\lim_{c \to \infty} \mathcal{W}^{(\omega)}(c)^{-1}\mathcal{Z}^{(\omega)}(c)
%=\lim_{c \to \infty} \left(\left(e^{\mathbf{\Lambda}^{\lambda}c}\mathcal{W}^{(\omega)}(c)\right)^{-1}e^{\mathbf{\Lambda}^{\lambda}c}\mathcal{Z}^{(\omega)}(c)\right)
=\mathbf{C}_{\mathcal{W}(\infty)^{-1}\mathcal{Z}(\infty)}=\mathbf{C}_{B^{(\omega)}} $$
is by \eqref{CBomega} finite.
This completes the proof.
\halmos

\subsection{Proof of Theorem \ref{Res}}
 %The potential measure or resolvent is given
%\begin{align*}
%\boldsymbol{U}^{(\omega)}(x,dy):&=\int_0^{\infty} \mathbb{E}_{x}\left[\exp\left( -\int_0^{t} \omega_{J_s}(X_s)ds \right),  X_t\in dy,  t<\tau_0^- \wedge \tau_c^+, J_t|J_0 \right] dt\\
%&= \left ( \mathcal{W}^{(\omega)}(x) \mathcal{W}^{(\omega)}(c)^{-1} \mathcal{W}^{(\omega)}(c,y)-\mathcal{W}^{(\omega)}(x,y)\right) dy,
%\end{align*}
%for $0\le x\le c$ and $y\ge0$.
\subsubsection{Proof of the case $(i)$}
 Using Lemma \ref{lemk}, we have
%where $J_t=j$ for the $\omega$-type resolvent, we have
\begin{align}
\boldsymbol{U}^{(\omega)}_{(d,c)}\boldsymbol{f}(x):&=\int_0^{\infty} \mathbb{E}_{x}\left[f_{J_t}(X_t)\exp\left( -\int_0^{t} \omega_{J_s}(X_s)ds \right), t<\tau_d^- \wedge \tau_c^+, J_t|J_0 \right] dt \nonumber\\
&=\int_d^{c}  \boldsymbol{U}_{(d,c)}(x,dy) \left ( \boldsymbol{f}(y)-\boldsymbol{\omega}(y)\boldsymbol{U}_{(d,c)}^{(\omega)}\boldsymbol{f}(y)  \right), \label{Uwf}
\end{align}
where $\boldsymbol{U}_{(d,c)}(x,dy)$ is the potential measure of the MAP without $\omega$-killing, given in Theorem 1 of Ivanovs \cite{I14}:
\[
\boldsymbol{U}_{(d,c)}(x,dy)=\left( \mathbf{W}(x-d)\mathbf{W}(c-d)^{-1}\mathbf{W}(c-y)-\mathbf{W}(x-y) \right)dy.
\]

Hence, we can rewrite Equation (\ref{Uwf}) as
\[
\boldsymbol{U}^{(\omega)}_{(d,c)}\boldsymbol{f}(x)= \mathbf{W}(x-d) \mathbf{C}_U -\int_d^x \mathbf{W}(x-y)\boldsymbol{f}(y) dy
+\int_d^x \mathbf{W}(x-y) \boldsymbol{\omega}(y)\boldsymbol{U}_{(d,c)}^{(\omega)}\boldsymbol{f}(y)dy,
\]
where
\[
\mathbf{C}_U=\int_d^c \mathbf{W}(c-d)^{-1}\mathbf{W}(c-y)  \left (  \boldsymbol{f}(y)- \boldsymbol{\omega}(y)\boldsymbol{U}_{(d,c)}^{(\omega)}\boldsymbol{f}(y)  \right)dy.
\]

Multiplying Equation (\ref{Woxy}) by $\mathbf{C}_U$ gives that
\[
\mathcal{W}^{(\omega)}(x,d)\mathbf{C}_U =\mathbf{W}(x-d)\mathbf{C}_U+\int_d^{x} \mathbf{W}(x-y)\boldsymbol{\omega}(y)\mathcal{W}^{(\omega)}(y,d)\mathbf{C}_U dy,
\]
and define the operator $\mathcal{R}^{(\omega)} \boldsymbol{f}(x):=\int_d^{x}\mathcal{W}^{(\omega)}(x,y)\boldsymbol{f}(y)dy$, which leads to
\[
\mathcal{R}^{(\omega)} \boldsymbol{f}(x)=\int_d^x \mathbf{W}(x-y)\boldsymbol{f}(y) dy+\int_{d}^{x} \mathbf{W}(x-y)\boldsymbol{\omega}(y)\mathcal{R}^{(\omega)} \boldsymbol{f}(y) dy.
%\mathcal{R}^{(\omega)} \boldsymbol{f}(x)=\int_d^x \mathbf{W}(x-y)\boldsymbol{f}(y) dy+\int_{d}^{x} \mathbf{W}(x-z)\boldsymbol{\omega}(z)\int_d^{z}\mathcal{W}^{(\omega)}(z,y) \boldsymbol{f}(y)dydz.
\]
Therefore, by the uniqueness property in Lemma \ref{lem unique}, we have
\[
\boldsymbol{U}_{(d,c)}^{(\omega)}\boldsymbol{f}(x)=\mathcal{W}^{(\omega)}(x,d)\mathbf{C}_U-\mathcal{R}^{(\omega)} \boldsymbol{f}(x).%\int_d^{x}\mathcal{W}^{(\omega)}(x,y)\boldsymbol{f}(y)dy.
\]
To find the constant matrix $\mathbf{C}_U$, we use the boundary condition $\boldsymbol{U}^{(\omega)}_{(d,c)}\boldsymbol{f}(c)=0$.
One completes the proof by denoting the density of $\boldsymbol{U}^{(\omega)}_{(d,c)}\boldsymbol{f}(x)$ as $\boldsymbol{U}^{(\omega)}_{(d,c)}(x,dy)$.
\halmos

\subsubsection{Proof of the case $(ii)$}
This identity follows directly from Theorem \ref{Res} $(i)$ by taking the limit and using (\ref{lim_W})
together with the dominated convergence theorem.
%\begin{align*}
%& \mathbf{C}_{\mathcal{W}(\infty)^{-1}\mathcal{W}(\infty)}(y)=\lim_{c \to \infty} \mathcal{W}^{(\omega)}(c)^{-1}\mathcal{W}^{(\omega)}(c,y)
%=\lim_{c \to \infty} \left(\left(e^{\mathbf{\Lambda}^{\lambda}c}\mathcal{W}^{(\omega)}(c)\right)^{-1}e^{\mathbf{\Lambda}^{\lambda}c}\mathcal{W}^{(\omega)}(c,y)\right)\\
%&=\left(\mathbf{L}^{\lambda}+\int_{0}^{\infty} e^{\Lambda^{\lambda} z}
%\mathbf{L}^{\lambda}(\boldsymbol{\omega}(z)-\lambda\mathbf{I})\mathcal{W}^{(\omega)}(z) dz\right)^{-1}
%\left(e^{\mathbf{\Lambda}^{\lambda} y}\mathbf{L}^{\lambda}+\int_{y}^{\infty} e^{\mathbf{\Lambda}^{\lambda} z} \mathbf{L}^{\lambda}(\boldsymbol{\omega}(z)-\lambda\mathbf{I})\mathcal{W}^{(\omega)}(z,y) dz\right).
%\end{align*}
\halmos

\subsubsection{Proof of the case $(iii)$}
The formula follows by taking the limit $\lim_{d \to -\infty}$ in Theorem \ref{Res} $(i)$ and then using (\ref{one-sided_(i)}).
\halmos

\subsubsection{Proof of the case $(ii)$}
This identity follows from Theorem \ref{Res} $(iii)$ by taking the limit $c \to \infty$.
Since $\mathcal{H}^{(\omega)}(c)^{-1}\mathcal{W}^{(\omega)}(c,y)$ is monotonic of $c$ then the result holds.
\halmos

\section*{Acknowledgements}
I. Czarna is partially supported by the National Science Centre Grant No. 2015/19/D/ST1/01182.
A. Kaszubowski is partially supported by the National Science Centre Grant \linebreak No. 2015/17/B/ST1/01102.
Z. Palmowski is partially supported by the National Science Centre Grant No. 2016/23/B/HS4/00566.
S. Li acknowledges the support from a start-up grant from the University of Illinois at Urbana-Champaign.

\end{document}